\documentclass[11pt]{amsart}
\usepackage{amsmath}
\usepackage{amssymb}
\usepackage{amscd}
\usepackage[all]{xy}
\usepackage{xypic}
\usepackage{comment}
\usepackage{graphicx}
\usepackage[capitalize]{cleveref}

\numberwithin{equation}{section}

\def\today{\number\day\space\ifcase\month\or   January\or February\or
March\or April\or May\or June\or   July\or August\or September\or
October\or November\or December\fi\   \number\year}

\newtheorem{Thm}{Theorem}[section]
\newtheorem{Cor}[Thm]{Corollary}
\newtheorem{Lemma}[Thm]{Lemma}
\newtheorem{Prop}[Thm]{Proposition}

\theoremstyle{definition}
\newtheorem{Def}[Thm]{Definition}

\newtheorem{qst}[Thm]{Question}
\newtheorem{rmk}[Thm]{Remark}
\newtheorem{cnv}[Thm]{Convention}
\newtheorem{ntn}[Thm]{Notation}
\newtheorem{exa}[Thm]{Example}

\newcounter{TmpEnumi}

\newcommand{\limi}[1]{\lim_{{#1} \to \infty}}

\newcommand{\Q}{{\mathbb{Q}}}
\newcommand{\Z}{{\mathbb{Z}}}
\newcommand{\R}{{\mathbb{R}}}
\newcommand{\C}{{\mathbb{C}}}
\newcommand{\N}{{\mathbb{N}}}
\newcommand{\Nz}{{\mathbb{Z}}_{\geq 0}}

\newcommand{\card}{\operatorname{card}}
\newcommand{\Aut}{\operatorname{Aut}}
\newcommand{\Ad}{{\operatorname{Ad}}}

\newcommand{\Cu}{{\operatorname{Cu}}}
\newcommand{\spec}{\operatorname{sp}}
\newcommand{\ev}{\operatorname{ev}}
\newcommand{\At}{\operatorname{At}}

\newcommand{\Zh}{\mathcal{Z}}

\newcommand{\fr}{*_{\operatorname{r}}}
\newcommand{\id}{\operatorname{id}}

\newcommand{\eps}{\varepsilon}
\numberwithin{equation}{section}

\newcommand{\dirlim}{\varinjlim}

\newcommand{\bigastr}{\mathop{\scalebox{1.7}{$\;*$}_r}\limits}
\newcommand{\bigast}{\mathop{\scalebox{1.7}{$\;*$}\,\,}\limits}

\newcommand{\af}{\alpha}
\newcommand{\bt}{\beta}
\newcommand{\gm}{\gamma}
\newcommand{\dt}{\delta}
\newcommand{\ep}{\varepsilon}
\newcommand{\zt}{\zeta}
\newcommand{\et}{\eta}
\newcommand{\ch}{\chi}
\newcommand{\io}{\iota}
\newcommand{\te}{\theta}
\newcommand{\ld}{\lambda}
\newcommand{\sm}{\sigma}
\newcommand{\kp}{\kappa}
\newcommand{\ph}{\varphi}
\newcommand{\ps}{\psi}
\newcommand{\rh}{\rho}
\newcommand{\om}{\omega}
\newcommand{\ta}{\tau}

\newcommand{\Sm}{\Sigma}

\newcommand{\cT}{{\mathcal{T}}}

\newcommand{\andeqn}{\qquad {\mbox{and}} \qquad}

\newcommand{\sandeqn}{\quad {\mbox{and}} \quad}

\newcommand{\Wolog}{Without loss of generality}

\newcommand{\ifo}{if and only if}

\newcommand{\ca}{C*-algebra}
\newcommand{\uca}{unital C*-algebra}

\newcommand{\hm}{homomorphism}
\newcommand{\uhm}{unital homomorphism}

\newcommand{\fd}{finite dimensional}
\newcommand{\tst}{tracial state}

\newcommand{\pj}{projection}

\newcommand{\ct}{continuous}

\newcommand{\cms}{compact metric space}

\newcommand{\I}{\infty}
\renewcommand{\S}{\subset}
\newcommand{\SQ}{\subseteq}
\newcommand{\E}{\varnothing}

\title[An isomorphism theorem for infinite reduced free
 products]{An isomorphism theorem for infinite reduced free products}

\author{Ilan Hirshberg}
\address{Department of Mathematics, Ben-Gurion University of the Negev,
P.O.B. 653, Be'er Sheva 84105, Israel}

\author{N. Christopher Phillips}
\address{Department of Mathematics, University  of Oregon,
Eugene OR 97403-1222, USA.}

\subjclass[2020]{46L09}
\thanks{This material is based upon work supported by the
  US National Science Foundation under
   Grant DMS-2400332, and by the
 US-Israel Binational Science Foundation.}

\date{11~June 2026}

\begin{document}

\begin{abstract}
Let $(C_n)_{n \in {\mathbb{N}}}$ be a sequence
of separable unital C*-algebras, equipped with
faithful tracial states and satisfying a mild condition.
Let $A$ be a unital direct limit of one dimensional NCCW complexes,
also equipped with a faithful tracial state.
Suppose there is a unital trace preserving embedding of $A$
in the Jiang-Su algebra which is an isomorphism on K-theory.
(For example, $A$ could be $C ([0, 1])$ with Lebesgue measure,
or the Jiang-Su algebra itself.)
Let $D$ be the infinite reduced free product
of the algebras $C_n$, amalgamated over~${\mathbb{C}}$.
Then the reduced free product $A \fr D$,
amalgamated over~${\mathbb{C}}$, is isomorphic to~$D$.

If $D$ is exact and the factors satisfy
a blockwise real rank zero condition,
then in place of $A$ we can use
$C (X)$ for any contractible compact metric space $X$
and any faithful tracial state on $C (X)$.

An example consequence is that the reduced free product
$C ([0, 1])^{\fr \infty}$, with Lebesgue measure,
is isomorphic to $\Zh^{\fr \infty}$.
\end{abstract}

\maketitle

\section{Introduction}\label{sec_5Z17_Intro}

We prove isomorphism theorems for infinite reduced free products.
These theorems can be regarded as very basic cases of
classification of infinite reduced free products
in terms of their Elliott invariants in the nonnuclear setting.
Specifically, if $D$ is the infinite reduced free product
of a sequence of
separable unital \ca{s} $C_n$ satisfying very mild conditions,
and $A$ is a suitable direct limit of one dimensional
NCCW complexes,
in particular with $K_0 (A) = \Z \cdot [1_A]$ and $K_1 (A) = 0$,
we show that the reduced free product $A \fr D$ is isomorphic to $D$.
If the factors satisfy
a blockwise real rank zero condition,
in particular, when $C_n$ has real rank zero for every~$n$,
and $D$ is exact,
we may also take $A = C (X)$ for a contractible \cms~$X$.
As a concrete example, we get the isomorphism
\[
C ([0, 1])^{\fr \infty} \cong \Zh^{\fr \infty},
\]
using Lebesgue measure on $[0, 1]$
and the unique tracial state on the Jiang-Su algebra~$\Zh$.
The C*-algebras in our theorems are not nuclear and not $\Zh$-stable.

The von Neumann algebra prototype is
\cite[Theorem~1.5 (in Section~2)]{DkmRdl}:
if $A_1, A_2, \ldots$ are arbitrary factors of type $\mathrm{II}_1$
with separable preduals, then
\[
L (F_{\infty}) \ast \bigast_{n = 1}^{\I} A_n
 \cong \bigast_{n = 1}^{\I} A_n.
\]
That is, an infinite tracial free product of factors absorbs an extra
free factor of $L (F_{\infty})$.
In particular, using Lebesgue measure on
$[0, 1]$, it also absorbs a diffuse abelian free factor:
\[
L^{\infty} ([0, 1]) \ast \bigast_{n = 1}^{\I} A_n
 \cong \bigast_{n = 1}^{\I} A_n.
\]

Reduced free products were constructed by Avitzour~\cite{Avitzour}.
Given unital \ca{s} $A$ and $B$ with specified states $\et$ and~$\om$,
their unital reduced free product is
\[
 (A, \et) \fr (B, \om) = (A \fr B,\, \et \fr \om),
\]
where $\et \fr \om$ is the free product state.
Here, and throughout, \ca{s} are unital, and reduced free products are
amalgamated over $\C \cdot 1$.
Infinite reduced free products are obtained as direct limits of the
corresponding finite reduced free products; they already appear in
Section~4 of~\cite{Avitzour}.
We denote the $n$-fold reduced free product of
$(A_1,\omega_1), (A_2,\omega_2), \ldots$ by
$\bigastr_{k = 1}^{n} \, (A_k, \om_k)$, including the case $n = \infty$,
and omit the states from the notation when they are understood.
If all the factors are copies of the same pair $(A, \om)$, we write
$(A, \om)^{\fr n}$, or just $A^{\fr n}$.
For context, recall that if reduced \ca{s} of discrete groups are
equipped with their standard tracial states, then
\[
C_{\operatorname{r}}^{*} (G) \fr C_{\operatorname{r}}^{*} (H)
  \cong C_{\operatorname{r}}^{*} (G * H),
\]
and that $C (S^1)^{\fr n}$, with Haar measure on~$S^1$, is isomorphic
to $C_{\operatorname{r}}^{*} (F_n)$.

Using the notation above, our main theorems say that
\[
A \fr \bigastr_{n=1}^{\infty} C_n
   \cong \bigastr_{n=1}^{\infty} C_n
\]
under the following two sets of hypotheses.
\begin{enumerate}
\item\label{I_5Z17_Imt}
$A$ is a unital countable direct limit of one dimensional NCCW complexes
(Definition~2.4.1 of~\cite{EilrLPd}; also see~\cite[2.2]{Robert_NCCW})
equipped with a faithful \tst~$\rh$, and admits a trace preserving
embedding into~$\Zh$ which is an isomorphism on K-theory.
The algebras $C_1, C_2, \ldots$ are separable unital \ca{s}, equipped
with faithful \tst{s}~$\sm_n$,
and their tracial states do not concentrate too rapidly on atoms,
in a sense made precise in \cref{thm_main}.
\item\label{I_5Z17_RRZ}
$A \cong C (X)$ for a contractible \cms~$X$, equipped with a faithful
\tst~$\rh$.
The algebras $C_1, C_2, \ldots$ are separable unital exact \ca{s},
equipped with faithful \tst{s}~$\sm_n$, and satisfy the following
additional condition.
There is a partition of $\N$ into subsets $I_1, I_2, \ldots$ such that,
for every $m \in \N$, the set $I_m$ is infinite, the atom condition in
(\ref{I_5Z17_Imt}) holds along $I_m$, and the subgroup generated by the
sets $(\sm_k)_* (K_0 (C_k))$, for $k \in I_m$, is dense in~$\R$.
\end{enumerate}

These results can be regarded as a first step towards classification of
infinite reduced free products by their Elliott invariants.
The computations of the relevant invariants use existing results.
The K-theory is obtained by combining~\cite{Thn2} and~\cite{Hsgw}, and
\cite[Corollary on page~431]{Avitzour} together with
\cite[Proposition~6.3.2]{Robert_NCCW} gives unique tracial state and
strict comparison of positive elements.
These computations imply that the natural inclusion
\[
 \bigastr_{n = 1}^{\infty} C_n
   \to A \fr \bigastr_{n = 1}^{\infty} C_n
\]
is an isomorphism on the Elliott invariants.

The algebras in our theorems are never nuclear;
this follows from \cite[Theorem~4.6]{Dykm2}
and the fact that free group factors are not injective.
Moreover, in~(\ref{I_5Z17_Imt}), if at least one of the free factors
$C_n$ is not exact, then the resulting free product is not exact.
The algebras are also never $\Zh$-stable.
It is conceivable, although our methods do little towards this,
that, under suitable assumptions on the states,
in particular that infinitely many of them are not pure,
two infinite reduced free products of
purely infinite simple separable nuclear unital \ca{s} satisfying the
Universal Coefficient Theorem are isomorphic whenever they have isomorphic
K-theory.
We give further discussion of the purely infinite case starting before
\cref{P_6302_UCT}.

There are clear limits to any classification statement
based only on K-theoretic data.
Simply tensoring with~$\Zh$ gives a nonisomorphic \ca.
Real rank zero does not help.
See \cref{Ex_6601_UHF}.
More dramatically, at least for finite reduced free products,
one can tensor with a UHF algebra and form the crossed product
by an extremely well behaved action of a finite cyclic group
(of a kind known to preserve both $\Zh$-stability
and the Universal Coefficient Theorem), but get
an algebra not isomorphic to its opposite algebra.
See \cref{Ex_6603_Opp} and the discussion afterwards.

We have no results for finite free products.
Here, the von Neumann algebraic analogs are understood:
\cite{Dykm2} determines, up to isomorphism, free
products of finite dimensional von Neumann algebras and hyperfinite
von Neumann algebras of type $\mathrm{II}_1$ with separable preduals,
with respect to faithful normal tracial states, modulo the still open
question of whether the factors associated to different free groups are
isomorphic.

Our proofs use Elliott approximate intertwining arguments.
The required uniqueness comes from uniqueness theorems, up to
approximate unitary equivalence, for \hm{s} from countable direct limits
of one dimensional NCCW complexes to \ca{s} with stable rank
one~\cite{Robert_NCCW}, and from $C (X)$ to infinite dimensional simple
separable unital exact \ca{s} with real rank zero, stable rank one,
weakly unperforated~$K_0$, and finitely many extreme tracial
states~\cite{Matui}.
We do not prove uniqueness theorems for \hm{s} from algebras such as
$C ([0, 1]) \fr C ([0, 1])$, and such theorems are very unlikely
to hold in a form useful here.
Indeed, the free flip on $C ([0, 1]) \fr C ([0, 1])$ induces the identity
map on the Elliott invariant but is not approximately inner.
To see this, note that the weak closure under the
Gelfand-Naimark-Segal representation associated with the trace is the
free group factor $L (F_2)$.
The free flip on $L (F_2)$ is not inner
(\cite[Example~2.5(1)]{Kallman}), and since $L (F_2)$ does not have
property~$\Gamma$, all approximately inner automorphisms of $L (F_2)$
in the von Neumann algebraic sense are inner
(\cite[Corollary~3.8]{Connes}).
For a more striking example, let $D$ be the closed unit disk, equipped
with normalized Lebesgue measure.
For $z \in S^1$, let $\beta_z$ be the automorphism of
$C(D) \fr C([0, 1])$ induced by rotating the disk by~$z$ and fixing the
free factor $C([0, 1])$.
One can show that these automorphisms act trivially on the Elliott
invariant and on Hausdorffized algebraic $K_1$, but are pairwise not
approximately unitarily equivalent.

Approximate unitary equivalence is used
in Elliott approximate intertwining
arguments because they require knowing, for certain
unital \hm{s} $\ph \colon A \to B$, that if $\af \in \Aut (A)$, then
there is $\bt \in \Aut (B)$ such that $\bt \circ \ph = \ph \circ \af$.
When $\af$ is conjugation by a unitary $u \in A$,
then $\bt$ can be taken to be conjugation by $\ph (u)$.
In our proofs we have $B \cong A \fr D$ for some~$D$, and the relevant
automorphism is $\bt = \af \fr \id_D$, which is not inner,
but does extend~$\af$.

This paper is organized as follows.
In the rest of the introduction, we collect some notation.
Section~\ref{sec_preliminaries} contains preliminaries on one dimensional
NCCW complexes, reduced free products, and the form of Elliott
approximate intertwining we use.
The main theorems, together with corollaries and examples, are in
Section~\ref{Sec_5Y27_Main}.
In Section~\ref{Sec_5Z17_Open}, we give examples to show
some limitations on possible generalizations,
and discuss a number of open problems.

We take $\N = \{ 1, 2, \ldots \}$.
If $A$ is a unital \ca{} and $u \in A$ is unitary, then $\Ad (u)$ is the
automorphism $a \mapsto u a u^*$.
% The Jiang-Su algebra is denoted by~$\Zh$.

\begin{ntn}\label{Nt_5Z19_FP}
As discussed above, all reduced free products are amalgamated over
$\C \cdot 1$.
The reduced free product of $n$~copies, possibly with $n = \infty$, of
the same unital \ca{} $A$ with the same state $\om$ is abbreviated to
$(A, \om)^{\fr n}$.

When the states used in the construction of a reduced free product are
understood, they are omitted from the notation, giving, for example,
$A^{\fr n}$.
Unless otherwise specified, the state used on~$\Zh$ is its unique \tst,
and the state used on $C ([0, 1])$ is Lebesgue measure.
\end{ntn}

\begin{cnv}\label{Cv_5Z19_DLim}
All direct systems are indexed by~$\N$.
We specify them in the form
$\bigl( ( A_n)_{n \in \N}, ( \alpha_{n + 1, \, n})_{n \in \N} \bigr)$
or the form
$\bigl( ( A_n)_{n \in \N}, ( \alpha_{n, m})_{n \geq m} \bigr)$
as convenient.
In the second form, we require
$\af_{n, m} \circ \alpha_{m, l} = \alpha_{n, l}$,
and the maps are related to the first form by
$\af_{n, m}
  = \af_{n, \, n - 1} \circ \af_{n - 1, \, n - 2} \circ \cdots
    \circ \af_{m + 1, \, m}$.
In either form, following the notation of the second form, we let
$\af_{\infty, n}$ denote the canonical map from $A_n$ to the direct
limit.
\end{cnv}

We thank Miles Gould, Adrian Ioana, Narutaka Ozawa, Leonel Robert,
and Hannes Thiel for some pointers and comments.

\section{Preliminaries}\label{sec_preliminaries}

This section contains preliminary results needed for the proof
of the main theorem.

We record the following simple fact.
It is well known and follows immediately from the definition
of reduced free products, so we omit the proof.

\begin{Lemma}\label{lem_extending_maps_along_free_products}
Let $(A_1, \tau_1)$, $(A_2, \tau_2)$, $(B_1, \sigma_1)$,
and $(B_2, \sigma_2)$ be unital C*-algebras with given states.
Suppose that for $j=1, 2$ we have unital homomorphisms
$\varphi_j \colon A_j \to B_j$
such that $\sigma_j \circ \varphi_j = \tau_j$.
Then these homomorphisms uniquely define a homomorphism
\[
\varphi_1 \fr \varphi_2 \colon (A_1, \tau_1) \fr (A_2, \tau_2)
\to (B_1, \sigma_1) \fr (B_2, \sigma_2).
\]
\end{Lemma}

The following theorem is the basic setup we need
for the Elliott intertwining argument.
It differs from most versions in the literature
in that we do not assume, for example,
that $\af_{n, \, n - 1} (\Xi_{n - 1}) \subseteq \Xi_n$;
instead, we only require approximate containment.
In the proof of \cref{thm_main},
this allows us to take our finite subsets to be in the algebraic
free product.
The statement and proof are well known,
but we haven't found it stated this way in the literature,
so we provide full details for the reader's convenience.

\begin{Thm}\label{thm_Elliott_intertwining}
Let $\bigl( ( A_n)_{n \in \N}, ( \alpha_{m, n})_{m \geq n} \bigr)$
and $\bigl( ( B_n)_{n \in \N}, ( \beta_{m, n})_{m \geq n} \bigr)$
be two direct systems of separable C*-algebras.
Write
\[
A = \dirlim A_n \andeqn B = \dirlim B_n.
\]
Let $\alpha_{\infty, n} \colon A_n \to A$
and $\beta_{\infty, n} \colon B_n \to B$
be the canonical homomorphisms to the direct limits
(\cref{Cv_5Z19_DLim}).
Suppose that for each $n \in \N$ we are given
\hm{s} $\mu_n \colon A_n \to B_n$ and $\nu_n \colon B_n \to A_{n + 1}$,
finite sets $\Xi_n \subset A_n$ and $\Sm_n \subset B_n$, and $\ep_n > 0$.
Assume:
\begin{enumerate}
\item\label{5X23_Summ}
$\sum_{n = 1}^{\I} \ep_n < \I$ and the sequence $(\ep_n)_{n \in \N}$ is 
decreasing.
\setcounter{TmpEnumi}{\value{enumi}}
\end{enumerate}
For every $n \in \N$, assume (taking $\Xi_0 = \E$ and $\Sm_0 = \E$):
\begin{enumerate}
\setcounter{enumi}{\value{TmpEnumi}}
\item\label{5X23_Xin}
For every $a \in A_n$, with the limit taken over $k \geq n$,
\[
\lim_{k \to \infty}
 \min \bigl( \{ \| x - \alpha_{k, n} (a) \| \mid x \in \Xi_k \} \bigr) = 0.
\]
\item\label{5X23_Smn}
For every $b \in B_n$, with the limit taken over $k \geq n$,
\[
\lim_{k \to \infty}
 \min \bigl( \{ \| y - \bt_{k, n} (b) \| \mid y \in \Sm_k \} \bigr) = 0.
\]
\item\label{5X23_phnA}
For every $x \in \Xi_{n - 1}$ there is ${\widetilde{x}} \in \Xi_{n}$
such that
$\bigl\| \af_{n, \, n - 1} (x) - {\widetilde{x}} \bigr\| < \ep_n$.
\item\label{5X23_phn_2}
For every $x \in \Xi_n$ there is $y \in \Sm_n$ such that
$\| \mu_n (x) - y \| < \ep_n$.
\item\label{5X23_psinB}
For every $y \in \Sm_{n - 1}$ there is ${\widetilde{y}} \in \Sm_{n}$
such that $\bigl\| \bt_{n, \, n - 1} (y) - {\widetilde{y}} \bigr\| < \ep_n$.
\item\label{5X23_psin_2}
For every $y \in \Sm_{n - 1}$ there is $x \in \Xi_{n}$ such that
$\| \nu_{n-1} (y) - x \| < \ep_n$.
\item\label{5X23_Acomm}
For every $x \in \Xi_n$,
we have $\| (\nu_n \circ \mu_n) (x) - \alpha_{n + 1, \, n} (x) \| < \eps_n$.
\item\label{5X23_SmD}
For every $y \in \Sm_n$,
we have
$\| (\mu_{n + 1} \circ \nu_n) (y) - \beta_{n + 1, \, n} (y) \| < \eps_n$.
\setcounter{TmpEnumi}{\value{enumi}}
\end{enumerate}
Then there is a unique \hm{} $\mu \colon A \to B$
such that for all $n \in \N$ and for all $x \in \Xi_n$
we have
\[
\mu (\af_{\I, n} (x))
 = \lim_{k \to \infty}
   (\beta_{ \infty, k} \circ \mu_k \circ \alpha_{k, n} ) (x),
\]
with the limit taken over $k \geq n$.
Moreover, $\mu$ is an isomorphism.
\end{Thm}

The diagram is as follows:
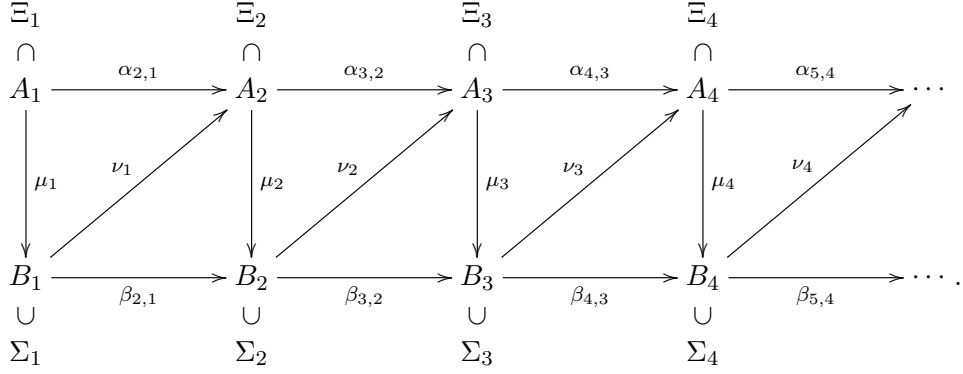
\[
\begin{xy}
(0,25)*+{A_1 }="top_1";
(30,25)*+{ A_2  }="top_2";
(60,25)*+{ A_3 }="top_3";
(90,25)*+{ A_4 }="top_4";
(120,25)*+{ \cdots }="top_5";
(0,0)*+{ B_1 }="bottom_1";
(30,0)*+{ B_2 }="bottom_2";
(60,0)*+{  B_3 }="bottom_3";
(90,0)*+{ B_4 }="bottom_4";
(120,0)*+{ \cdots }="bottom_5";
{\ar^-{\alpha_{2, 1}} "top_1";"top_2"};
{\ar^-{\alpha_{3, 2}} "top_2";"top_3"};
{\ar_-{\beta_{2, 1}} "bottom_1";"bottom_2"};
{\ar_-{\beta_{3, 2}} "bottom_2";"bottom_3"};
{\ar^-{\alpha_{4, 3}} "top_3";"top_4"};
{\ar_-{\beta_{4, 3}} "bottom_3";"bottom_4"};
{\ar^-{\alpha_{5, 4}} "top_4";"top_5"};
{\ar_-{\beta_{5, 4}} "bottom_4";"bottom_5"};
{\ar^-{\nu_1} "bottom_1";"top_2"};
{\ar^-{\nu_2} "bottom_2";"top_3"};
{\ar^-{\nu_3 } "bottom_3";"top_4"};
{\ar^-{\nu_4 } "bottom_4";"top_5"};
{\ar^-{\mu_1} "top_1";"bottom_1"};
{\ar^-{\mu_2} "top_2";"bottom_2"};
{\ar^-{\mu_3} "top_3";"bottom_3"};
{\ar^-{\mu_4} "top_4";"bottom_4"};
(0,30)*+{\rotatebox{270}{$\subset$}};
(30,30)*+{\rotatebox{270}{$\subset$}};
(60,30)*+{\rotatebox{270}{$\subset$}};
(90,30)*+{\rotatebox{270}{$\subset$}};
(0,35)*+{\Xi_1};
(30,35)*+{\Xi_2};
(60,35)*+{\Xi_3};
(90,35)*+{\Xi_4};
(0,-5)*+{\rotatebox{90}{$\subset$}};
(30,-5)*+{\rotatebox{90}{$\subset$}};
(60,-5)*+{\rotatebox{90}{$\subset$}};
(90,-5)*+{\rotatebox{90}{$\subset$}};
(0,-10)*+{\Sigma_1};
(30,-10)*+{\Sigma_2};
(60,-10)*+{\Sigma_3};
(90,-10)*+{\Sigma_4};
\end{xy}.
\]

\begin{proof}[Proof of \cref{thm_Elliott_intertwining}]
We begin by deriving the following consequences of the estimates
in the hypotheses:
\begin{enumerate}
\setcounter{enumi}{\value{TmpEnumi}}
\item\label{I_5Y25_mnXi}
For every $m, n \in \N$ with $n > m$, and for every $x \in \Xi_m$,
there is ${\widetilde{x}} \in \Xi_n$ such that
$\bigl\| \af_{n, m} (x) - {\widetilde{x}} \bigr\|
 < \sum_{k = m + 1}^{n} \ep_k$.
\item\label{I_5Y25_mnSm}
For every $m, n \in \N$ with $n > m$, and for every $y \in \Sm_m$,
there is ${\widetilde{y}} \in \Sm_n$ such that
$\bigl\| \bt_{n, m} (y) - {\widetilde{y}} \bigr\|
 < \sum_{k = m + 1}^{n} \ep_k$.
\item\label{I_5Y25_pabp}
For every $n \in \N$ and $x \in \Xi_n$, we have
\[
\bigl\| (\mu_{n + 1} \circ \af_{n + 1, \, n} ) (x)
  - (\bt_{n + 1, \, n} \circ \mu_n) (x) \bigr\| < 4 \ep_n.
\]
\item\label{I_5Y25_Gap_A_B}
For every $m, n \in \N$ with $n \geq m$, and for every $x \in \Xi_m$,
we have
\[
\bigl\| (\mu_{n} \circ \af_{n, m} ) (x)
     - (\bt_{n, m} \circ \mu_m) (x) \bigr\|
  < 6 \sum_{k = m}^{n - 1} \ep_k.
\]
\item\label{I_5Y25_Gap2_A_A}
For every $m, n \in \N$ with $n \geq m$, and for every $x \in \Xi_m$,
we have
\[
\bigl\| (\nu_{n} \circ \bt_{n, m} \circ \mu_m) (x)
     - \af_{n + 1, \, m} (x) \bigr\|
  < 6 \ep_m + 8 \sum_{k = m + 1}^{n} \ep_k.
\]
\item\label{I_5Y25_Gap3_B_A}
For every $m, n \in \N$ with $n \geq m$, and for every $y \in \Sm_m$,
we have
\[
\bigl\| (\nu_{n} \circ \bt_{n, m}) (y)
     - (\af_{n + 1, \, m + 1} \circ \nu_m ) (y) \bigr\|
  < \ep_m + 8 \sum_{k = m + 1}^{n} \ep_k.
\]
\item\label{I_5Y25_Gap4_B_B}
For every $m, n \in \N$ with $n \geq m + 1$, and for every $y \in \Sm_m$,
we have
\[
\bigl\| (\mu_{n} \circ \af_{n, \, m + 1} \circ \nu_m) (y)
     - \bt_{n, m} (y) \bigr\|
  < 8 \sum_{k = m}^{n - 1} \ep_k.
\]
\end{enumerate}
(The estimates (\ref{I_5Y25_Gap2_A_A}), (\ref{I_5Y25_Gap3_B_A}),
and~(\ref{I_5Y25_Gap4_B_B})
can be improved, but the forms given suffice for our purposes
and are easily derived from~(\ref{I_5Y25_Gap_A_B}).)

The statements (\ref{I_5Y25_mnXi}) and~(\ref{I_5Y25_mnSm})
follow from (\ref{5X23_phnA}) and~(\ref{5X23_psinB}) by induction.

For~(\ref{I_5Y25_pabp}), use~(\ref{5X23_phn_2}) to
choose $y \in \Sm_n$ such that $\| \mu_n (x) - y \| < \ep_n$.
Then, using~(\ref{5X23_Acomm}) and~(\ref{5X23_SmD})
at the second step,
\[
\begin{split}
&
\| (\mu_{n + 1} \circ \af_{n + 1, \, n} ) (x)
  - (\bt_{n + 1, \, n} \circ \mu_n) (x) \|
\\
& \hspace*{3em} {\mbox{}}
  \leq \bigl\| \mu_{n + 1}
    \bigl( \af_{n + 1, \, n} (x) - (\nu_{n} \circ \mu_n) (x) \bigr) \bigr\|
\\
& \hspace*{6em} {\mbox{}}
   + 2 \| \mu_n (x) - y \|
   + \| (\mu_{n + 1} \circ \nu_n) (y) - \bt_{n + 1, \, n} (y) \|
\\
& \hspace*{3em} {\mbox{}}
  < \ep_n + 2 \ep_n + \ep_n
  = 4 \ep_n.
\end{split}
\]

We prove~(\ref{I_5Y25_Gap_A_B}).
Set $x_m = x$.
Use~(\ref{5X23_phnA}) to inductively choose $x_k \in \Xi_k$,
for $k = m + 1, \, m + 2, \, \ldots, \, n - 1$,
such that
\begin{equation}\label{Eq_5Y26_kkm}
\| \af_{k, \, k - 1} (x_{k - 1}) - x_k \| < \ep_k.
\end{equation}
By~(\ref{I_5Y25_pabp}), we have
\[
\bigl\| (\mu_{k} \circ \af_{k, \, k - 1} ) (x_{k - 1})
  - (\bt_{k, \, k - 1} \circ \mu_{k - 1}) (x_{k - 1}) \bigr\|
  < 4 \ep_{k - 1}.
\]
Now, using $x = x_m$ at the second step,
and at the third step repeatedly applying (\ref{I_5Y25_pabp})
and~(\ref{Eq_5Y26_kkm}),
\[
\begin{split}
& \bigl\| (\mu_{n} \circ \af_{n, m} ) (x)
     - (\bt_{n, m} \circ \mu_m) (x) \bigr\|
\\
& \hspace*{2em} {\mbox{}}
\leq \sum_{k = m + 1}^{n}
  \bigl\| \bt_{n, k} \bigl( [\mu_{k} \circ \af_{k, \, k - 1} ]
          ( \af_{k - 1, m} (x))
   - [\bt_{k, \, k - 1} \circ \mu_{k-1}] ( \af_{k - 1, m} (x)) \bigr) \bigr\|
\\
& \hspace*{2em} {\mbox{}}
 \leq \| \bt_{n, \, m + 1} \|
     \bigl\| (\mu_{m + 1} \circ \af_{m + 1, \, m} ) (x_m)
  - (\bt_{m + 1, \, m} \circ \mu_m) (x_m) \bigr\|
\\
& \hspace*{4em} {\mbox{}}
  + \sum_{k = m + 1}^{n - 1}
            \Bigl[ 2 \| \af_{k, \, k - 1} (x_{k - 1}) - x_k \|
 \\
& \hspace*{6em} {\mbox{}}
       + \| \bt_{n, \, k + 1} \|
          \bigl\| (\mu_{k + 1} \circ \af_{k + 1, \, k} ) (x_{k})
       - (\bt_{k + 1, \, k} \circ \mu_{k}) (x_{k}) \bigr\|
     \Bigr]
\\
& \hspace*{2em} {\mbox{}}
 < 4 \ep_m + \sum_{k = m + 1}^{n - 1} (2 \ep_{k} + 4 \ep_{k})
 \leq 6 \sum_{k = m}^{n - 1} \ep_k,
\end{split}
\]
as wanted.

Next, we prove~(\ref{I_5Y25_Gap2_A_A}).
Use~(\ref{I_5Y25_mnXi})
to choose ${\widetilde{x}} \in \Xi_n$ such that
$\bigl\| \af_{n, m} (x) - {\widetilde{x}} \bigr\|
 < \sum_{k = m + 1}^{n} \ep_k$.
Then,
using~(\ref{I_5Y25_Gap_A_B}) and~(\ref{5X23_Acomm}) at the second step,
\[
\begin{split}
& \bigl\| (\nu_{n} \circ \bt_{n, m} \circ \mu_m) (x)
     - \af_{n + 1, \, m} (x) \bigr\|
\\
& \hspace*{3em} {\mbox{}}
\leq \| \nu_n \| \bigl\| (\mu_{n} \circ \af_{n, m} ) (x)
     - (\bt_{n, m} \circ \mu_m) (x) \bigr\|
\\
& \hspace*{6em} {\mbox{}}
   + 2 \| \af_{n, m} (x) - {\widetilde{x}} \|
     + \bigl\| (\nu_n \circ \mu_n) ({\widetilde{x}})
           - \alpha_{n + 1, \, n} ({\widetilde{x}}) \bigr\|
\\
& \hspace*{3em} {\mbox{}}
  < 6 \sum_{k = m}^{n - 1} \ep_k + 2 \sum_{k = m + 1}^{n} \ep_k + \ep_n
  \leq 6 \ep_m + 8 \sum_{k = m + 1}^{n} \ep_k,
\end{split}
\]
as required.

Now we prove~(\ref{I_5Y25_Gap4_B_B}).
Use~(\ref{5X23_psin_2}) to choose
$x \in \Xi_{m + 1}$ such that $\| x - \nu_m (y) \| < \ep_{m + 1}$.
Using~(\ref{I_5Y25_Gap_A_B}) and~(\ref{5X23_SmD}) at the second step,
\[
\begin{split}
& \bigl\| (\mu_{n} \circ \af_{n, \, m + 1} \circ \nu_m) (y)
     - \bt_{n, \, m} (y) \bigr\|
\\
& \hspace*{3em} {\mbox{}}
  \leq 2 \| x - \nu_m (y) \|
     + \bigl\| (\mu_{n} \circ \af_{n, \, m + 1} ) (x)
     - (\bt_{n, \, m + 1} \circ \mu_{m + 1}) (x) \bigr\|
\\
& \hspace*{6em} {\mbox{}}
      + \bigl\| \bt_{n, \, m + 1}
         \bigl( \mu_{m + 1} \circ \nu_{m}) (y)
         - \bt_{m + 1, \, m} (y) \bigr) \bigr\|
\\
& \hspace*{3em} {\mbox{}}
  < 2 \ep_{m + 1} + 6 \sum_{k = m + 1}^{n - 1} \ep_k + \ep_m
  \leq 8 \sum_{k = m}^{n - 1} \ep_k.
\end{split}
\]

The estimate~(\ref{I_5Y25_Gap3_B_A})
is derived from~(\ref{I_5Y25_Gap2_A_A})
in the same way that (\ref{I_5Y25_Gap4_B_B})
is derived from~(\ref{I_5Y25_Gap_A_B}).

Now define
\[
\Xi = \bigcup_{n = 1}^{\I} \af_{\I, n} (\Xi_n) \subseteq A
\andeqn
\Sm = \bigcup_{n = 1}^{\I} \bt_{\I, n} (\Sm_n) \subseteq B.
\]
We claim that $\Xi$ is dense in~$A$ and $\Sm$ is dense in~$B$.
We prove the first; the proof of the second is the same,
using~(\ref{5X23_Smn}) instead of~(\ref{5X23_Xin}).
Let $a \in A$ and let $\ep > 0$.
Choose $m \in \N$ and $x \in A_m$ such that
$\| \af_{\I, m} (x) - a \| < \frac{\ep}{2}$.
Use~(\ref{5X23_Xin}) to choose $n \geq m$
and ${\widetilde{x}} \in \Xi_n$ such that
$\| {\widetilde{x}} - \alpha_{n, m} (x) \| < \frac{\ep}{2}$.
Then $\af_{\I, n} ({\widetilde{x}}) \in \Xi$
and $\| \af_{\I, n} ({\widetilde{x}}) - a \| < \ep$.
This proves the claim.

Uniqueness of $\mu$ in the statement
is immediate from density of~$\Xi$.

For $n \in \N$ and $k \geq n$, define
\begin{equation}\label{Eq_5Y26_mut_Dfn}
{\widetilde{\mu}}_n^{(k)}
 = \bt_{\I, k} \circ \mu_k \circ \af_{k, n} \colon A_n \to B.
\end{equation}
Then, whenever $k \geq m \geq n$, we have
\begin{equation}\label{Eq_5Y24_NewSt}
{\widetilde{\mu}}_m^{(k)} \circ \af_{m, n} = {\widetilde{\mu}}_n^{(k)}.
\end{equation}

We claim that if $n \in \N$ and $x \in A_n$, then
$\limi{k} {\widetilde{\mu}}_n^{(k)} (x)$ exists.
We prove the claim by showing that
$\bigl( {\widetilde{\mu}}_n^{(k)} (x) \bigr)_{k \geq n}$
is a Cauchy sequence.
Let $\ep > 0$.
Using~(\ref{5X23_Summ}) and~(\ref{5X23_Xin}),
choose $n_0 \in \N$ so large that
\begin{equation}\label{Eq_5Y26_z1}
\sum_{k = n_0}^{\I} \ep_k < \frac{\ep}{10}
\end{equation}
and such that there is $x_0 \in \Xi_{n_0}$ such that
\begin{equation}\label{Eq_5Y26_Z2}
\| \af_{n_0, n} (x) - x_0 \| < \frac{\ep}{10}.
\end{equation}
Let $l, m \in \N$ satisfy $l, m \geq n_0$.
\Wolog{} $l \leq m$.
It follows from (\ref{Eq_5Y26_z1}) and~(\ref{I_5Y25_mnXi})
that there is $x_1 \in \Xi_l$ such that
$\| x_1 - \af_{l, n_0} (x_0) \| < \frac{\ep}{10}$.
By~(\ref{Eq_5Y26_Z2}), we get
\begin{equation}\label{Eq_5Y26_stsar}
\| x_1 - \af_{l, n} (x) \| < \frac{\ep}{5}.
\end{equation}
Now, using (\ref{Eq_5Y26_stsar}) and~(\ref{I_5Y25_Gap_A_B})
at the second step, and~(\ref{Eq_5Y26_z1}) at the third step,
\[
\begin{split}
\bigl\| {\widetilde{\mu}}_n^{(m)} (x)
  - {\widetilde{\mu}}_n^{(l)} (x) \bigr\|
& \leq \bigl\| (\mu_m \circ \af_{m, l} \circ \af_{l, n} ) (x)
    - (\bt_{m, l} \circ \mu_l \circ \af_{l, n} ) (x) \bigr\|
\\
& \leq 2 \| x_1 - \af_{l, n} (x) \|
  + \bigl\| (\mu_m \circ \af_{m, l}) (x_1)
    - (\bt_{m, l} \circ \mu_l) (x_1) \bigr\|
\\
& < \frac{2 \ep}{5} + 6 \sum_{k = l}^{m - 1} \ep_k
  \leq \frac{2 \ep}{5} + \frac{3 \ep}{5}
  = \ep.
\end{split}
\]
This proves the claim.

We now claim that there is a \ct{} \hm{}
\[
\mu \colon \bigcup_{n = 1}^{\I} \af_{\I, n} (A_n) \to B
\]
such that for all $n \in \N$ and $x \in A_n$, we have
\begin{equation}\label{Eq_5Y25_StSt}
(\mu \circ \af_{\I, n}) (x) = \limi{k} {\widetilde{\mu}}_n^{(k)} (x).
\end{equation}
To prove the claim, first, use~(\ref{Eq_5Y24_NewSt})
to see that~(\ref{Eq_5Y25_StSt}) gives a well defined function
$\mu$ from $\bigcup_{n = 1}^{\I} \af_{\I, n} (A_n)$ to~$B$.
It is then immediate that $\mu$ is \ct{} and a \hm, as claimed.

Extending by continuity gives a \hm{} $\mu \colon A \to B$
satisfying~(\ref{Eq_5Y25_StSt}).
The claim is
then the condition on $\mu$ in the statement of the theorem.

A similar argument shows that
there is also a \hm{} $\nu \colon B \to A$ such that
for all $n \in \N$ and $y \in B_n$, we have
\begin{equation}\label{Eq_5Y26_nu_dfn}
(\nu \circ \bt_{\I, n}) (y)
  = \limi{k} (\af_{\I, \, k + 1} \circ \nu_k \circ \bt_{k, n}) (y).
\end{equation}

We now prove that $\nu \circ \mu = \id_A$.
It suffices to prove that for every
$n \in \N$ and $x \in A_n$, we have
$(\nu \circ \mu \circ \af_{\I, n}) (x) = \af_{\I, n} (x)$.
Let $\ep > 0$; we prove that
\begin{equation}\label{Eq_5Y26_Lep}
\| (\nu \circ \mu \circ \af_{\I, n}) (x) - \af_{\I, n} (x) \| < \ep.
\end{equation}
By~(\ref{5X23_Summ}), there is $n_0 \geq n$ such that
\begin{equation}\label{Eq_6122_Star}
\sum_{k = n_0 + 1}^{\I} \ep_k < \frac{\ep}{12}.
\end{equation}
Since $\Xi$ is dense in~$A$
and the sets $\af_{\I, m} (\Xi_m)$ are finite,
the set $\bigcup_{m = n_0}^{\I} \af_{\I, m} (\Xi_m)$
is also dense in~$A$.
Therefore there are
$n_1 \geq n_0$ and $x_1 \in \Xi_{n_1}$
such that
\begin{equation}\label{Eq_5Y26_NewSt}
\| \af_{\I, n_1} (x_1) - \af_{\I, n} (x) \| < \frac{\ep}{12}.
\end{equation}

Recalling (\ref{Eq_5Y26_mut_Dfn}) and~(\ref{Eq_5Y25_StSt}),
choose $l \geq n_1$ such that
\[
\bigl\| (\mu \circ \af_{\I, n_1}) (x_1)
  - ( \bt_{\I, l} \circ \mu_l \circ \af_{l, n_1}) (x_1) \bigr\|
 < \frac{\ep}{12}.
\]
Using~(\ref{Eq_5Y26_nu_dfn}), choose $m \geq l$ such that,
with $y = (\mu_l \circ \af_{l, n_1}) (x_1) \in B_l$,
we have
\[
\bigl\| (\nu \circ \bt_{\I, l}) (y)
  - (\af_{\I, \, m + 1} \circ \nu_m \circ \bt_{m, l}) (y) \bigr\|
 < \frac{\ep}{12}.
\]
Then
\begin{equation}\label{Eq_5Y26_NwStSt}
\bigl\| (\nu \circ \mu \circ \af_{\I, n_1}) (x_1)
   - (\af_{\I, \, m + 1} \circ \nu_m \circ \bt_{m, l}
           \circ \mu_l \circ \af_{l, n_1}) (x_1) \bigr\|
  < \frac{\ep}{6}.
\end{equation}
By~(\ref{I_5Y25_mnXi}), there is ${\widetilde{x}} \in \Xi_l$ such that
$\| \af_{l, n_1} (x_1) - {\widetilde{x}} \|
 < \sum_{k = n_1 + 1}^{l} \ep_k$.
Now, using this and (\ref{I_5Y25_Gap2_A_A}) at the second step,
and (\ref{Eq_6122_Star}) at the last step,
\[
\begin{split}
& \bigl\| (\af_{\I, \, m + 1} \circ \nu_m \circ \bt_{m, l}
           \circ \mu_l \circ \af_{l, n_1}) (x_1)
        - \af_{\I, n_1} (x_1) \bigr\|
\\
& \hspace*{3em} {\mbox{}}
 \leq 2 \| \af_{l, n_1} (x_1) - {\widetilde{x}} \|
       + \bigl\| \af_{\I, \, m + 1} \bigl( (\nu_m \circ \bt_{m, l}
                \circ \mu_l) ({\widetilde{x}})
           - \af_{m + 1, \, l} ({\widetilde{x}}) \bigr) \bigr\|
\\
& \hspace*{3em} {\mbox{}}
 < 2 \sum_{k = n_1 + 1}^{l} \ep_k
    + 6 \ep_l + 8 \sum_{k = l + 1}^{m} \ep_k
 \leq 8 \sum_{k = n_1 + 1}^{l} \ep_k + 8 \sum_{k = l + 1}^{m} \ep_k
 < \frac{2 \ep}{3}.
\end{split}
\]
Combining this with~(\ref{Eq_5Y26_NewSt}) (twice)
and~(\ref{Eq_5Y26_NwStSt}) gives~(\ref{Eq_5Y26_Lep}).
This completes the proof that $\nu \circ \mu = \id_A$.

A similar argument,
using (\ref{I_5Y25_mnSm}) and~(\ref{I_5Y25_Gap4_B_B})
in place of (\ref{I_5Y25_mnXi}) and~(\ref{I_5Y25_Gap2_A_A}),
shows that $\mu \circ \nu = \id_B$.
Therefore $\mu$ is an isomorphism.
\end{proof}

\begin{Lemma}\label{lem_isomorphism_different_systems_0}
Let $\bigl( ( A_n )_{n \in \N}, ( \alpha_{m, n} )_{m \geq n} \bigr)$
be a direct system of C*-algebras.
Suppose that for $n = 2, 3, \ldots$
we are given an automorphism $\gamma_n \in \Aut(A_n)$.
Suppose furthermore that for every $k, m \in \N$ with $m \geq k > 1$,
there is $\beta_m^{(k)} \in \Aut(A_m)$ such that:
\begin{enumerate}
\item\label{I_5X23_Init}
For $m = 2, 3, \ldots$, we have $\beta_m^{(m)} = \gm_m$.
\item\label{I_5X23_Rec}
For $m = 2, 3, \ldots$ and $k = 2, 3, \ldots, m - 1$, we have
\[
\beta_m^{(k)} \circ \af_{m, \, m - 1}
 = \af_{m, \, m - 1} \circ \beta_{m - 1}^{(k)}.
\]
\end{enumerate}
Then
\[
\dirlim \bigl(  ( A_n )_{n \in \N}, (\alpha_{n + 1, \, n} )_{n \in \N}  \bigr)
 \cong \dirlim \bigl( ( A_n )_{n \in \N},
     (\gamma_{n + 1} \circ \alpha_{n + 1, \, n} )_{n \in \N} \bigr).
\]
\end{Lemma}

\begin{proof}
The following commutative diagram
gives an isomorphism between these two direct systems:
\[
\begin{xy}
(0,25)*+{A_1 }="top_1";
(30,25)*+{ A_2  }="top_2";
(60,25)*+{ A_3 }="top_3";
(90,25)*+{ A_4 }="top_4";
(120,25)*+{ \cdots }="top_5";
(0,0)*+{ A_1 }="bottom_1";
(30,0)*+{ A_2 }="bottom_2";
(60,0)*+{  A_3 }="bottom_3";
(90,0)*+{ A_4 }="bottom_4";
(120,0)*+{ \cdots. }="bottom_5";
{\ar^-{\alpha_{2, 1}} "top_1";"top_2"};
{\ar^-{\alpha_{3, 2}} "top_2";"top_3"};
{\ar_-{\gamma_{2} \circ \alpha_{2, 1}} "bottom_1";"bottom_2"};
{\ar_-{\gamma_{3} \circ \alpha_{3, 2}} "bottom_2";"bottom_3"};
{\ar^-{\id_{A_1}} "top_1";"bottom_1"};
{\ar^-{\beta_2^{(2)}} "top_2";"bottom_2"};
{\ar^-{\beta_3^{(3)} \circ \beta_3^{(2)} } "top_3";"bottom_3"};
{\ar^-{\beta_4^{(4)} \circ \beta_4^{(3)} \circ \beta_4^{(2)} }
       "top_4";"bottom_4"};
{\ar^-{\alpha_{4, 3}} "top_3";"top_4"};
{\ar_-{\gamma_4 \circ \alpha_{4, 3}} "bottom_3";"bottom_4"};
{\ar^-{\alpha_{5, 4}} "top_4";"top_5"};
{\ar_-{\gamma_5 \circ \alpha_{5, 4}} "bottom_4";"bottom_5"};
\end{xy}
\]
This completes the proof.
\end{proof}

\begin{Cor}\label{lem_isomorphism_different_systems}
Let $(A_1, \tau_1), (A_2, \tau_2), (A_3, \tau_3), \ldots$
be unital C*-algebras with given states.
Let $\gamma_n  \in \Aut \left( \bigastr_{k = 1}^{n} (A_k, \tau_k) \right)$
be automorphisms which fix the canonical state
$\tau_1 \fr \tau_2 \fr \cdots \fr \tau_n$.
Let
\[
\iota_{n} \colon \bigastr_{k = 1}^{n} (A_k, \tau_k)
 \to \bigastr_{k = 1}^{n + 1} (A_k, \tau_k)
\]
be the inclusion in the first $n$ free factors.
Then
\[
\dirlim
 \left( \left( \bigastr_{k = 1}^{n} (A_k, \tau_k) \right)_{n \in \N},
             (\gamma_{n + 1} \circ \iota_n )_{n \in \N} \right)
 \cong \dirlim
     \left( \left( \bigastr_{k = 1}^{n} (A_k, \tau_k) \right)_{n \in \N},
        (\iota_n)_{n \in \N} \right).
\]
\end{Cor}

\begin{proof}
We use \cref{lem_isomorphism_different_systems_0},
with $\bigastr_{k = 1}^{n} (A_k, \tau_k)$ in place of~$A_n$,
with $\gm_n$ as given, and with $\af_{n + 1, \, n} = \io_n$.
For $m = 2, 3, \ldots$, define $\beta_m^{(m)} = \gm_m$.
Using \cref{lem_extending_maps_along_free_products} to see that
the required free products of \hm{s} exist,
for fixed $k \in \{ 2, 3, \ldots \}$,
for $m \geq k$ inductively define
$\beta_{m + 1}^{(k)} = \beta_{m}^{(k)} \fr \id_{A_{m + 1}}$.
\end{proof}

\begin{Def}\label{def_At}
Let $A$ be a \uca{} and let $\ta$ be a faithful \tst{} on $A$.
For $a\in A_{\mathrm{sa}}$, let $\mu_{\ta,a}$
be the Borel probability
measure on $\spec(a)$ determined by
\[
\ta(f(a))=\int_{\spec(a)} f(t)\,d\mu_{\ta,a}(t)
\]
for all $f \in C (\spec (a))$.
Define
\[
\At(A,\ta)
=
\inf_{a\in A_{\mathrm{sa}}}
\sup_{t\in \spec(a)} \mu_{\ta,a}(\{t\}).
\]
\end{Def}

\begin{Lemma}\label{lemma_Infinite_Free_Product_Selfless}
For $n\in \N$, let $C_n$ be a separable \uca{} equipped with a
faithful \tst{} $\sm_n$. Suppose that
\[
\sum_{n=1}^{\infty}
\bigl(1-\At(C_n,\sm_n)\bigr)=\infty .
\]
Let
\[
(D,\rh)=\bigastr_{n=1}^{\infty} (C_n,\sm_n).
\]
Then $D$ is simple, has a unique tracial state, has stable rank one,
has strict comparison, and admits a unital embedding from the Jiang-Su
algebra $\Zh$.
\end{Lemma}

The decomposition in~(\ref{Eq_6524_Decomp}) in the proof below
will also be used in the proof
of \cref{C_2622_RanK0}(\ref{I_C_2622_RanK0_RRZ}).

\begin{proof}[Proof of
 Lemma~\ref{lemma_Infinite_Free_Product_Selfless}]
For each $n \in \N$, choose $a_n\in (C_n)_{\mathrm{sa}}$ such that,
writing
\[
\mu_n=\mu_{\sm_n,a_n}
\andeqn
\alpha_n=\sup_{t\in \spec(a_n)}\mu_n(\{t\}),
\]
we have
$\alpha_n < \At(C_n,\sm_n)+2^{-n}$.
Then $\sum_{n=1}^{\infty}(1-\alpha_n)=\infty$.
Thus we may write $\N$ as a disjoint union of finite nonempty sets
$S_1,S_2,S_3,\ldots$, such that for every $j \in \N$ we have
\[
\sum_{k\in S_j}(1-\alpha_k)>1.
\]

For $j \in \N$, set
$(D_j,\rh_j)=\bigastr_{k\in S_j}(C_k,\sm_k)$.
So
\begin{equation}\label{Eq_6524_Decomp}
(D,\rh)\cong \bigastr_{j=1}^{\infty}(D_j,\rh_j).
\end{equation}
We claim that $(D_j,\rh_j)$ admits a Haar unitary for each~$j$.
Fix $j$, and for $k \in S_j$ let
\[
\lambda_k\colon C_k\to D_j
\]
be the canonical embedding.
Set
\[
b_j=\sum_{k\in S_j}\lambda_k(a_k)\in (D_j)_{\mathrm{sa}}.
\]
By free independence,
the spectral distribution of $b_j$ with respect to $\rh_j$ is
given by the free additive convolution
(\cite[Definition 3.1.1]{VoiDyNi})
\[
\mu_{\rh_j,b_j}
=
\boxplus_{k\in S_j}\mu_k .
\]

The atom formula for free additive
convolution, \cite[Theorem~7.4]{Bercovici_Voiculescu_Regularity}, implies
by induction on $\card(S_j)$ that if $\boxplus_{k\in S_j}\mu_k$ has an
atom,
then there are atoms $t_k$ of $\mu_k$, for $k\in S_j$, such that
\[
\sum_{k\in S_j}\mu_k(\{t_k\})>\card(S_j)-1.
\]
But
\[
\sum_{k\in S_j}\mu_k(\{t_k\})
\leq
\sum_{k\in S_j}\alpha_k
=
\card(S_j)-\sum_{k\in S_j}(1-\alpha_k)
<
\card(S_j)-1 .
\]
Therefore $\mu_{\rh_j,b_j}$ is atomless.
Thus the restriction of $\rh_j$ to the unital abelian subalgebra
$C^*(b_j) \subseteq D_j$ is diffuse.
By \cite[Proposition~4.9]{Thiel_diffuse}, $\rh_j$ admits a Haar unitary.

Theorem~2.8 of~\cite{Robert_selfless} and~(\ref{Eq_6524_Decomp})
now imply that $D$ is selfless
in the sense of \cite[Definition~2.8]{Robert_selfless}.
Since $\rh$ has a faithful Gelfand-Naimark-Segal representation,
\cite[Theorem~3.1]{Robert_selfless} implies that
$D$ is simple, $\rh$ is the unique tracial state on $D$, $D$ has
stable rank one and strict comparison, and $\rh$ is the unique
$2$-quasitracial state on $D$.
Since $D$ contains a Haar unitary, $D$ is infinite dimensional,
hence not isomorphic to $K (H)$ for any Hilbert space~$H$.
It follows from \cite[Proposition~6.3.1]{Robert_NCCW} that $\Zh$ embeds
unitally into $D$.
\end{proof}

\begin{rmk}\label{ER_6524_AtConmd}
The theorem quoted in the proof of
\cref{lemma_Infinite_Free_Product_Selfless},
\cite[Theorem~3.1]{Robert_selfless}, adds new information,
as well as packing together some earlier results: simplicity
and existence of a unique \tst{} follow from
\cite[Part (3) of the Corollary on page 431]{Avitzour},
and stable rank one follows from \cite[Theorem
3.8]{Dykema_Haagerup_Rordam}. (The correction
\cite{Dykema_Haagerup_Rordam_correction}
does not affect this result.)
\end{rmk}

The condition
\begin{equation}\label{Eq_6524_InfSum}
\sum_{n = 1}^{\infty} \bigl( 1 - \At (C_n, \sm_n) \bigr) = \infty
\end{equation}
from the hypotheses of
\cref{lemma_Infinite_Free_Product_Selfless} is very mild.

\begin{exa}\label{Ex_6524_CplusC}
Let $\lambda_1, \lambda_2, \ldots \in (0, \frac{1}{2}]$
and take $C_n = \C \oplus \C$,
with $\sigma_n (a, b) = \lambda_n a + (1-\lambda_n) b$.
Then (\ref{Eq_6524_InfSum})
is $\sum_{n=1}^{\infty} \lambda_n = \infty$.
\end{exa}

\begin{exa}\label{Ex_6524_Mn}
If $C_n = M_{r (n)}$ with its
unique trace $\tau_n$,
then $\At (C_n, \tau_n) = 1 / r (n)$.
Therefore (\ref{Eq_6524_InfSum}) holds as long as $r (n) > 1$
for infinitely many~$n$.
\end{exa}

\begin{exa}\label{Ex_6524_Const}
If $C$ is a simple unital and infinite dimensional \ca{} with a
\tst~$\sigma$, then $C$ has a Haar unitary $u$ by
\cite[Corollary 5.6]{Thiel_diffuse}.
Then $a = u + u^*$ is selfadjoint
and its spectral measure with respect to $\sm$ has no atoms,
so $\At(C, \sigma) = 0$.
Thus, if all the factors are simple, unital, and infinite dimensional,
then (\ref{Eq_6524_InfSum}) holds.
\end{exa}

\begin{Lemma}\label{L_2622_KthGen}
Let $J$ be a finite or countable index set.
Let $( (A_j, \rh_j))_{j \in J}$ be a family of separable unital \ca{s}
$A_j$ equipped with faithful states~$\rh_j$.
Let $A = \bigastr_{j \in J} (A_j, \rh_j)$ be the reduced free product
of the algebras $A_j$ with respect to the states~$\rh_j$,
amalgamated over~$\C$.
For $j \in J$ let $\io_j \colon A_j \to A$ be the canonical inclusion.
Then $K_0 (A)$ is generated by $\bigcup_{j \in J} (\io_j)_* (K_0 (A_j))$.
If $A_j$ is stably finite for all $j \in J$,
then $K_1 (A)$ is generated by $\bigcup_{j \in J} (\io_j)_* (K_1 (A_j))$.
\end{Lemma}

\begin{proof}
For the full free product in place of~$A$,
and if $J = \{ 1, 2 \}$,
the statement for $K_0 (A)$
is immediate from the case $B = \C$ of \cite[Theorem 6.4]{Thn2}.
When at least one of $A_1$ and $A_2$ is stably finite,
then the map $K_0 (\C) \to K_0 (A_1) \oplus K_0 (A_2)$ in
\cite[Theorem 6.4]{Thn2} is injective,
so the statement for $K_1 (A)$ also follows.
For finite $J$, $K_0 (A)$, and full free products, use induction.
Induction also gives the statement for $K_1 (A)$,
since in the induction step one of the algebras is stably finite.
For infinite~$J$, take direct limits.
The result for the reduced free product now follows from
the case $B = \C$ of \cite[Theorem 1.1]{Hsgw}.
\end{proof}

\begin{Cor}\label{C_2622_RanK0}
For each $n \in \N$, let $C_n \not\cong \C$ be a separable \uca{}
with a faithful \tst{} $\sm_n$.
Suppose
$\sum_{n=1}^{\infty} \bigl(1-\At(C_n,\sm_n)\bigr)=\infty$.
Set $D=\bigastr_{n=1}^{\infty}(C_n,\sm_n)$,
and let $\ta$ be the unique \tst{} on $D$
(\cref{lemma_Infinite_Free_Product_Selfless}).
Then:
\begin{enumerate}
\item\label{I_C_2622_RanK0_Ran}
$\ta_* (K_0 (D))$ is the subgroup of $\R$ generated by
$\bigcup_{n = 1}^{\infty} (\sm_n)_* (K_0 (C_n))$.
\item\label{I_C_2622_RanK0_RRZ}
The algebra $D$ has real rank zero if and only if
$\bigcup_{n = 1}^{\infty} (\sm_n)_* (K_0 (C_n))$
generates a dense subgroup of~$\R$.
\end{enumerate}
\end{Cor}

\begin{proof}
We prove~(\ref{I_C_2622_RanK0_Ran}). For $n \in \N$, let
\[
\io_n\colon C_n\to D
\]
be the inclusion of $C_n$ as the $n$-th free factor.
Then $\ta \circ \io_n = \sm_n$ for all $n \in \N$.
By \cref{L_2622_KthGen}, $K_0(D)$ is generated by
$\bigcup_{n=1}^{\infty}(\io_n)_*(K_0(C_n))$.
Applying $\ta_*$, and using
$\ta_*\circ(\io_n)_*=(\sm_n)_*$,
gives~(\ref{I_C_2622_RanK0_Ran}).

For~(\ref{I_C_2622_RanK0_RRZ}), use the proof of
\cref{lemma_Infinite_Free_Product_Selfless}: the algebra $D$ can be
written as an infinite reduced free product $D\cong
\bigastr_{j=1}^{\infty}(D_j,\rh_j)$
such that each $D_j$ contains a Haar unitary, and hence a unitary in
$\ker(\rh_j)$.
Therefore \cite[Theorem~2.1(iii)]{Dykema_Rordam_II}
applies to this decomposition,
and implies that $D$ has real rank zero if
and only if $\ta_*(K_0(D))$ is dense in~$\R$.
Now use~(\ref{I_C_2622_RanK0_Ran}).
\end{proof}

Recall that a one dimensional NCCW complex $A$
in the sense of \cite[2.2]{Robert_NCCW} is
the pullback in a diagram of the following form,
in which $E$ and $F$ are finite
dimensional C*-algebras, $\ev_t \colon C ([0, 1], \, F) \to F$ is
evaluation at $t \in [0, 1]$, and $\varphi$ is a homomorphism:
\[
\begin{CD}
A @>{\pi_2}>> C ([0, 1], \, F)   \\
@VV{\pi_1}V  @VV{(\ev_0, \ev_1)}V            \\
E @>{\ph}>> F \oplus F.
\end{CD}
\]
That is,
\[
A = \bigl\{ (x, f) \in E \oplus C ([0, 1], \, F) \colon
\ph (x) = ( f (0), f (1) ) \bigr\}.
\]
The following is a special case
of \cite[Theorem B]{abstract_classification}.

\begin{Prop}\label{prop_embedding_of_nccw}
Let $A$ be a unital direct limit of one dimensional NCCW complexes
which satisfies $K_0(A) \cong \Z \cdot [1_A]$ and $K_1 (A) = 0$.
Let $\rh$ be a faithful tracial state on $A$.
Then there exists a unital
embedding $\te \colon A \to \Zh$ such that
$\te_{*} \colon K_0 (A) \to K_0 (\Zh)$ is an isomorphism
and, with $\tau$ denoting the
unique tracial state on $\Zh$, we have $\tau \circ \te = \rh$.
\end{Prop}

It isn't enough to require that $A$ have no nontrivial \pj{s}.
As was pointed out to us by Leonel Robert, the algebra
\[
\bigl\{ f \in C ([0, 1], \, M_2) \colon
  {\mbox{$f (0)$ is diagonal and $f (1)$ is a scalar}} \bigr\}
\]
is a counterexample.

\begin{rmk}
It further follows from \cite[Theorem B]{abstract_classification}
that any two embeddings as in
\cref{prop_embedding_of_nccw} are approximately unitarily equivalent.
We need a stronger conclusion, allowing for nonnuclear codomains,
hence the
need to restrict the domain to be a direct limit
of one dimensional NCCW complexes.
See \cref{C_5Z18_Hom_on_Cu} below.
\end{rmk}

\section{The main theorems}\label{Sec_5Y27_Main}

In this section, we prove our main theorems,
give some interesting special cases as corollaries,
and give some examples.

We can view \cref{thm_main} and \cref{thm_main_RRZ}
as classification results in terms of the Elliott invariant,
for the following reason.
First, by combining \cite[Theorem 6.4]{Thn2} and \cite[Theorem 1.1]{Hsgw}
in the same manner as in the proof of \cref{L_2622_KthGen},
it is easy to show that the two algebras have the same K-theory.
Indeed, the obvious map $\bigastr_{n = 1}^{\infty} C_n \to A \fr \bigastr_{n =
1}^{\infty} C_n$
is an isomorphism on K-theory.
Using the fact that both algebras have a unique \tst{} and
\cite[Theorem 2.1(i)]{Dykema_Rordam_II}
(order on projections is determined by traces),
one sees that this map is an isomorphism of the Elliott invariants.

\begin{Thm}\label{thm_main}
Let $( C_n , \sigma_n )_{n \in \N}$ be a sequence of separable unital
C*-algebras equipped with prescribed faithful tracial states.
Suppose that, in the notation of \cref{def_At},
\begin{equation}\label{Ex_6524_Star}
\sum_{n=1}^{\infty} ( 1 - \At(C_n,\sigma_n) ) = \infty .
\end{equation}
Let $A$ be a unital direct limit of one dimensional NCCW complexes
which satisfies $K_0(A) \cong \Z \cdot [1_A]$ and $K_1 (A) = 0$.
Let $\rh$ be a faithful tracial state on $A$.
Then, with respect to the tracial states $\rho$ and
$\sm_1, \sm_2, \ldots$, we have
\[
A \fr \bigastr_{n = 1}^{\infty} C_n \cong\bigastr_{n = 1}^{\infty} C_n.
\]
\end{Thm}

\begin{Thm}\label{thm_main_RRZ}
Let $( C_n , \sigma_n )_{n \in \N}$ be a sequence of separable unital
exact C*-algebras equipped with prescribed faithful tracial states.
Suppose that there exists a partition of $\N$ into infinite
subsets $I_1,I_2,\ldots$ such that for any $n \in \N$,
the subgroup of $\R$ generated by
$\{ (\sm_k)_* (K_0 (C_k)) \mid k \in I_n \}$ is dense and
\begin{equation}\label{Ex_6524_StSt}
\sum_{k \in I_n} ( 1 - \At(C_k,\sigma_k) ) = \infty .
\end{equation}
Let $X$ be a contractible compact metric space
and let $\rho$ be a faithful tracial state on $C (X)$.
Then, with respect to the tracial states $\rho$ and
$\sm_1, \sm_2, \ldots$, we have
\[
C (X) \fr \bigastr_{n = 1}^{\infty}
 C_n \cong \bigastr_{n = 1}^{\infty} C_n
.
\]
\end{Thm}

The proof of \cref{thm_main_RRZ} is almost identical to that of
\cref{thm_main}.
In its proof, we therefore only indicate what changes.

\cref{Ex_6524_CplusC}, \cref{Ex_6524_Mn}, and especially
\cref{Ex_6524_Const} show that
the condition~(\ref{Ex_6524_Star}) in \cref{thm_main}
and the condition~(\ref{Ex_6524_StSt}) in \cref{thm_main_RRZ}
are very commonly satisfied.

Some condition on the algebras $(C_n, \sm_n)$
in \cref{thm_main} and \cref{thm_main_RRZ}
is necessary.
See \cref{Ex_6604_NonSmp} below.

It seems convenient to isolate the calculation leading
to \cref{C_5Z18_Hom_on_Cu}.
It is surely known, but we have not found a reference.
\cref{L_5Z18_Hom_on_Cu} seems potentially useful elsewhere.

\begin{Lemma}\label{L_5Z18_Hom_on_Cu}
Let $A$ and $B$ be unital \ca{s}.
Assume that $B$ is infinite dimensional, simple, separable, exact,
has a unique \tst~$\ta$, has strict comparison, has stable rank $1$,
and that $K_0 (B) = \Z \cdot [1_B]$.
Then any two unital \hm{s} $\ph, \ps \colon A \to B$ such that
$\ta \circ \ph = \ta \circ \ps$ induce the same maps
$\ph_{*}, \ps_{*} \colon \Cu (A) \to \Cu (B)$.
\end{Lemma}

The map is described in the proof.
The proof only uses the structure of $\Cu (B)$,
which, as is shown in the proof,
is necessarily isomorphic to $\Cu (\Zh)$.

\begin{proof}[Proof of \cref{L_5Z18_Hom_on_Cu}]
By \cite[Theorem A]{Thiel_diffuse}, $B$ has a Haar unitary, and therefore a
positive contraction $a$ such that the measure induced on the spectrum of
$a$
by the trace is Lebesgue measure. From this, it follows that the image of
$d_{\tau}$ applied to positive elements is $[0,1]$.
It now follows from \cite[Theorem~2.6]{BrnTms}
(see the definitions just before the theorem) that
$\Cu (B) \cong (0, \I] \amalg \Nz$ with the usual structure.
Write $R = (0, \I]$ and $N = \Nz$.
For the convenience of the reader, we describe the structure.
Give both $R$ and $N$ the usual semigroup structure and order.
Then $N$ is the Murray-von Neumann semigroup $\operatorname{V} (B)$.
The order on mixed pairs $x \in R$ and $y \in N$ is given by
$x \leq y$ \ifo{} the relation $x \leq y$ holds in~$\R$,
and $y \leq x$ \ifo{} the relation $y < x$ holds in~$\R$.
Addition for such pairs is the usual addition $x + y \in \R$,
taken to be in~$R$.

Set $\rh = \ta \circ \ph$.
We give a formula for $\ph_{*}$
entirely in terms of $\rh = \ta \circ \ph$.
Since also $\rh = \ta \circ \ps$, the lemma will follow.
Let $a \in (K \otimes A)_{+}$.
Let $\mu$ be the spectral measure on $[0, \I)$ induced by~$\rh$:
for any \ct{} $f \colon [0, \I) \to \C$,
$\rh (f (a)) = \int_{[0, \I)} f \, d \mu$.
This implies that $\ta (f (\ph (a))) = \int_{[0, \I)} f \, d \mu$
for all such~$f$.

First suppose that there is $\ep > 0$ such that
$\mu ((0, \ep)) = 0$.
Since $\ta$ is faithful, $\spec (\ph (a)) \cap (0, \ep) = \E$.
Therefore $\langle \ph (a) \rangle \in N$,
with the value being $d_{\ta} (\ph (a)) = d_{\rh} (a)$.
If, on the other hand, no such $\ep$ exists,
then $0$ is a limit point of $\spec ( \ph (a))$.
Thus, $\langle \ph (a) \rangle \in R$,
with the value again being $d_{\ta} (\ph (a)) = d_{\rh} (a)$.
\end{proof}

\begin{Cor}\label{C_5Z18_Hom_on_Cu}
Let $A$ and $B$ be unital \ca{s}.
Assume that $K_1 (A) = 0$
and that $A$ is a unital one dimensional NCCW~complex,
or is a unital countable direct limit of such algebras.
Assume that $B$ is simple, separable, exact, has stable rank one,
has a unique \tst~$\ta$, has strict comparison,
and that $K_0 (B) = \Z \cdot [1_B]$.
Let $\ph, \ps \colon A \to B$ be unital \hm{s} such that
$\ta \circ \ph = \ta \circ \ps$.
Then $\ph$ and $\ps$ are approximately unitarily equivalent.
\end{Cor}

\begin{proof}
The hypotheses of \cite[Theorem~1.0.1]{Robert_NCCW} are satisfied.
The uniqueness part of that theorem therefore applies:
if $\ph^{\sim}, \ps^{\sim} \colon \Cu^{\sim} (A) \to \Cu^{\sim} (B)$
are equal, then $\ph$ and $\ps$ are approximately unitarily equivalent.
Since $A$ is unital, by \cite[Theorem 3.2.2(i)]{Robert_NCCW} it is
enough to know that $\ph_{*}, \ps_{*} \colon \Cu (A) \to \Cu (B)$
are equal.
This equality follows from \cref{L_5Z18_Hom_on_Cu}.
\end{proof}

\begin{proof}[Proof of \cref{thm_main}]
We may choose a partition of $\N$ into a countable union of infinite
countable sets $I_1, I_2, I_3, \ldots$
such that for any $k \in \N$, we have
$\sum_{n \in I_k} ( 1 - \At(C_n,\sigma_n) ) = \infty$.
By commutativity and associativity of countable reduced free products,
we may replace the
original algebras $(C_n,\sigma_n)$ with the block reduced free products
$\bigastr_{k \in I_n} (C_k, \sigma_k)$ and then rename these block
algebras as $C_n$, with their free product traces denoted by $\sigma_n$.
Thus, using \cref{lemma_Infinite_Free_Product_Selfless}, we may assume,
without loss of generality, that for each~$n$, the C*-algebra $C_n$
is simple, has a unique tracial state $\sigma_n$, has stable rank one,
has strict comparison, and admits a unital embedding from the Jiang-Su
algebra $\Zh$.

We will use \cref{thm_Elliott_intertwining}, with
(recalling \cref{Nt_5Z19_FP})
\[
A_n = A \fr \bigastr_{k = 1}^{n-1} C_k
\andeqn
B_n = \bigastr_{k = 1}^{n} C_k.
\]
We choose an arbitrary decreasing sequence $( \eps_n )_{n = 1, 2, 3,
\ldots}$
of strictly positive numbers such that $\sum_{n = 1}^{\I} \ep_n < \I$.
The maps and finite sets will be defined below.

For any $m, n \in \N$ with $m < n$, we write
\[
\iota_m^{(n)} \colon C_m \to A \fr C_{n-1} \fr C_{n-2} \fr \cdots \fr  C_1
\]
for the canonical
embedding of $C_m$ in the corresponding free factor.
(The lower index is in the reverse of the order
in which the factors appear.)
For $m=n$, we use
the same notation, with $A$ in place of $C_{n}$:
\[
\iota_{n}^{(n)} \colon A \to A \fr C_{n-1} \fr C_{n-2} \fr \cdots \fr  C_1 .
\]
Likewise, for any $m, n \in \N$ with $m \leq n$, we write
\[
\lambda_m^{(n)} \colon C_m \to C_n \fr C_{n-1} \fr \cdots \fr  C_1
\]
for the
canonical
embedding of $C_m$ in the corresponding free factor.
% By slight abuse of notation,
% we use the notation $\iota_m^{(n)}$ regardless of
% which C*-algebras are in the domain and in the codomain;
% this shouldn't
% cause problems, as those will always be clearly specified.

As usual, let $\tau$ be the unique tracial state on $\Zh$.
Use \cref{prop_embedding_of_nccw} to choose a unital
embedding $\te \colon A \to \Zh$ such that $\tau \circ \te = \rh$.
Fix unital homomorphisms $\zeta_n \colon \Zh \to C_n$
for $ n= 1, 2, \ldots$,
which exist by the rearrangement in the first paragraph of the proof.
Since $\Zh$ and $C_n$ have unique tracial states, $\zeta_n$ is
trace preserving, and hence so is $\varphi_n$ below.
For $ n= 1, 2, \ldots$, define
\begin{equation}\label{Eq_5Y17_St}
\varphi_n = \zeta_n \circ \theta \colon A \to C_n.
\end{equation}

Given $n \in \N$, we denote by
\[
\chi_1^{(n)} \colon A \to A \fr C_n
\andeqn
\chi_2^{(n)} \colon C_n \to A \fr C_n
\]
the first and second factor embeddings.
We next prove the following claim:
\begin{enumerate}
\item\label{I_5Y17_aue}
For any $n \in \N$, the maps $\chi_1^{(n)}, \chi_2^{(n)} \circ \ph_n
\colon
A \to A \fr C_n$
are approximately unitarily equivalent.
\setcounter{TmpEnumi}{\value{enumi}}
\end{enumerate}
Let $\kp_1 \colon A \to A \fr \Zh$ and $\kp_2 \colon \Zh \to A \fr \Zh$
be the first and second factor embeddings.
The following commutative diagram shows the maps we use:
\begin{equation}\label{Eq_5Y18_Diag_for_aue}
\begin{xy}
(0,25)*+{A }="top_1";
(45,25)*+{ \Zh  }="top_2";
(90,25)*+{C_n}="top_3";
(0,0)*+{ A }="bottom_1";
(45,0)*+{ A \fr \Zh }="bottom_2";
(90,0)*+{  A \fr C_n }="bottom_3";
{\ar^-{\theta} "top_1";"top_2"};
{\ar^-{\zeta_n} "top_2";"top_3"};
{\ar_-{\kappa_1} "bottom_1";"bottom_2"};
{\ar_-{\id_A \fr \zeta_n} "bottom_2";"bottom_3"};
{\ar^-{\kappa_2} "top_2";"bottom_2"};
{\ar^-{\chi_2^{(n)}} "top_3";"bottom_3"};
{\ar@/^2pc/^-{\varphi_n} "top_1";"top_3"};
{\ar@/_2pc/_-{\chi_1^{(n)}} "bottom_1";"bottom_3"};
\end{xy}.
\end{equation}
It follows from \cite[Corollary 5.3 and Theorem 3.1]{Robert_selfless}
that $A \fr \Zh$ is simple,
has stable rank $1$, strict comparison, and a unique tracial state.
Combining \cite[Theorem 6.4]{Thn2} and \cite[Theorem 1.1]{Hsgw}
in the same manner as in the proof of \cref{L_2622_KthGen},
we see that $K_0 (A \fr \Zh) \cong \Z$ and $K_1 (A \fr \Zh) = 0$,
with the class of $1$ generating $K_0$.
Moreover, $A \fr \Zh$ is exact by \cite[Corollary~4.3]{DkmShk}.
So \cref{C_5Z18_Hom_on_Cu} implies that
$\kp_1$ and $\kp_2 \circ \te$ are approximately unitarily equivalent.
The claim is now immediate from~(\ref{Eq_5Y18_Diag_for_aue}) by
applying the map $\id_A \fr \zeta_n$.

For $n \in \N$, we write
\[
\beta_{n + 1, \, n} \colon  C_n \fr C_{n-1} \fr
\ldots \fr C_1 \to C_{n+1} \fr C_n \fr C_{n-1} \fr
\ldots \fr C_1
\]
for the canonical embedding in the last $n$ factors, and
\[
\alpha_{n + 1, \, n} \colon A \fr C_{n-1} \fr
\ldots \fr C_1 \to A \fr C_n \fr C_{n-1} \fr
\ldots \fr C_1
\]
for the canonical embedding gotten from embedding $A$ in the first free
factor and the other factors in the last $n-1$ factors; if $n=1$, then
$\alpha_{2,1} = \iota_2^{(2)} \colon A \to A \fr C_1$.
We also denote by
\begin{equation}\label{Eq_6524_iotan}
\iota_n \colon  C_n \fr C_{n-1} \fr
\ldots \fr C_1 \to A \fr  C_n \fr C_{n-1} \fr
\ldots \fr C_1
\end{equation}
the canonical embedding into the last $n$ factors.

For notational convenience, we write
\[
\bigastr_{n = 1}^{\infty} C_n
 = \dirlim_n \bigl( ( C_n \fr C_{n-1} \fr \cdots \fr C_1)_{n \in \N},
    ( \beta_{n + 1, \, n})_{n \in \N} \bigr)
\]
and
\[
A \fr \bigastr_{n = 1}^{\infty} C_n
 = \dirlim_n
  \bigl( ( A \fr C_{n-1} \fr C_{n-2} \fr \cdots \fr C_1 )_{n \in \N},
 ( \alpha_{n + 1, \, n})_{n \in \N} \bigr).
\]

We start the main part of the proof with the following diagram,
in which we abbreviate
\begin{equation}\label{Eq_6524_idn}
\id_n = \id_{C_n \fr C_{n-1} \fr \cdots \fr C_1},
\end{equation}
and which is not approximately commutative:
\[
\begin{xy}
(0,25)*+{A }="top_1";
(32,25)*+{ A \fr C_1  }="top_2";
(83,25)*+{ A \fr C_2 \fr C_1 }="top_3";
(120,25)*+{ \cdots }="top_4";
(0,0)*+{ C_1 }="bottom_1";
(32,0)*+{ C_2 \fr C_1  }="bottom_2";
(83,0)*+{ C_3 \fr C_2 \fr C_1 }="bottom_3";
(120,0)*+{ \cdots }="bottom_4";
{\ar^-{\alpha_{2,1}} "top_1";"top_2"};
{\ar^-{\alpha_{3, 2}} "top_2";"top_3"};
{\ar^-{\alpha_{4, 3}} "top_3";"top_4"};
{\ar_-{\beta_{2,1}} "bottom_1";"bottom_2"};
{\ar_-{\beta_{3, 2}} "bottom_2";"bottom_3"};
{\ar_-{\beta_{4, 3}} "bottom_3";"bottom_4"};
{\ar^-{\varphi_1} "top_1";"bottom_1"};
{\ar_-{\iota_1} "bottom_1";"top_2"};
{\ar^-{\varphi_2 \fr \id_1} "top_2";"bottom_2"};
{\ar_-{\iota_{2}} "bottom_2";"top_3"};
{\ar^-{\varphi_3 \fr \id_2} "top_3";"bottom_3"};
{\ar_-{\iota_{3}}  "bottom_3";"top_4"};
(24,8)*+{\circlearrowleft};
(68,8)*+{\circlearrowleft};
(112,8)*+{\circlearrowleft};
\end{xy}.
\]
The lower triangles commute, that is, for $n \in \N$ we have
\begin{equation}\label{Eq_6524_lcomm}
(\ph_{n + 1} \fr \id_n) \circ \io_n = \beta_{n + 1, \, n}.
\end{equation}
Our goal is to modify the maps in
the top row so as to make this diagram approximately commute.

We will construct recursively a new top row,
with maps $\widetilde{\alpha}_{n+1,n}$
in place of~$\alpha_{n+1,n}$,
so that the resulting direct system
satisfies the hypotheses of \cref{thm_Elliott_intertwining},
and such that the direct limit of the top row is the same.
In the application of \cref{thm_Elliott_intertwining},
we will take $\widetilde{\alpha}_{n+1,n}$
in place of~$\alpha_{n+1,n}$,
\[
\mu_n = \ph_n \fr \id_{n - 1} \text{ with } \mu_1 = \ph_1,
\andeqn
\nu_n = \io_{n}.
\]
For $m, n \in \N$ with $m \leq n$, following \cref{Cv_5Z19_DLim}
we define
\[
\widetilde{\alpha}_{n, m} \colon A \fr C_{m-1} \fr C_{m-2} \fr \cdots \fr
C_1 \to A \fr  C_{n-1} \fr C_{n-2} \fr \cdots \fr C_1
\]
and
\[
\bt_{n, m} \colon  C_{m} \fr C_{m-1} \fr \cdots \fr C_1  \to  C_{n} \fr
C_{n-1} \fr \cdots \fr C_1 ,
\]
and for $l \in \N$ we define
\[
\begin{split}
& \widetilde{\alpha}_{\I, l} \colon
A \fr  C_{l-1} \fr C_{l-2} \fr \cdots \fr C_1
\\
& \hspace*{3em} {\mbox{}}
\to \dirlim_n
  \bigl( (A \fr C_{n-1} \fr C_{n-2} \fr \cdots \fr C_1 )_{n \in \N},
\, (\widetilde{\alpha}_{n, m})_{m \leq n} \bigr)
\end{split}
\]
and
\[
\begin{split}
& \bt_{\I, l} \colon
C_{l} \fr C_{l-1} \fr \cdots \fr C_1
\\
& \hspace*{3em} {\mbox{}}
  \to \dirlim_n
  \bigl( (C_{n} \fr C_{n-1} \fr \cdots \fr C_1 )_{n \in \N},
     \, (\bt_{n, m})_{m \leq n} \bigr).
\end{split}
\]

For any field~$K$ (we will use $K = \C$ and $K = \Q [i]$),
and for $n \in \N$,
we let $K \langle t_1, t_2, \ldots, t_n \rangle$ be the $K$-algebra of
polynomials in the noncommuting variables $t_1, t_2, \ldots, t_n$,
and similarly with $K \langle t_1, t_2, \ldots \rangle$.
We make the obvious identifications
\[
K \langle t_1 \rangle
\SQ K \langle t_1, t_2 \rangle
\SQ K \langle t_1, t_2, t_3 \rangle
\SQ \cdots
\SQ K \langle t_1, t_2, \ldots \rangle,
\]
with
\[
K \langle t_1, t_2, \ldots \rangle
= \bigcup_{n = 1}^{\I} K \langle t_1, t_2, \ldots, t_n \rangle.
\]
We evaluate $q \in K \langle t_1, t_2, \ldots, t_n \rangle$
at elements $b_1, b_2, \ldots, b_n$ in a \uca~$B$ in the obvious way,
and write $q (b_1, b_2, \ldots, b_n)$
or $q \bigl( (b_k)_{1 \leq k \leq n} \bigr)$.

We will choose inductively sequences of finite sets
\begin{equation}\label{Eq_5Y24_EnFnIn}
E_1 \subseteq E_2 \subseteq \cdots \subseteq A
\end{equation}
and, for $j \in \N$,
\begin{equation}\label{Eq_6524_33b}
F_1^{(j)} \subseteq F_2^{(j)} \subseteq
     \cdots \subseteq C_j,
\end{equation}
numbers $k (n) \in \N$ with $k (1) \leq k (2) \leq \cdots$,
and finite sets
$Q_n \subseteq
 \Q [i] \bigl\langle t_1, t_2, \ldots, t_{k (n)} \bigr\rangle$
for $n \in \N$,
such that
\begin{equation}\label{Eq_5X26_Incr}
\bigcup_{n = 1}^{\I} Q_n = \Q [i] \langle t_1, t_2, \ldots \rangle,
\qquad
{\overline{\bigcup_{n = 1}^{\I} E_n}} = A,
\end{equation}
and, for $j \in \N$,
\begin{equation}\label{Eq_6524_34B}
{\overline{\bigcup_{n = 1}^{\I} F_n^{(j)}}} = C_j,
\end{equation}
such that for all $n \in \N$,
\begin{equation}\label{Eq_5X27_kngn}
k (n) \geq n,
\end{equation}
for $j = 1, 2, \ldots, n$,
recalling from~(\ref{Eq_5Y17_St}) that
$\ph_j = \zt_j \circ \te \colon A \to C_j$,
\begin{equation}\label{Eq_5X26_Q_inc}
\ph_j (E_n) \subseteq F_n^{(j)},
\end{equation}
and, for $n \geq 2$,
\begin{equation}\label{Eq_5X26_ph}
Q_{n - 1} \subseteq Q_{n}.
\end{equation}

Given such choices, which are made later in the proof,
for $m = 1, 2, \ldots, n$
we define finite subsets
\[
\Xi_n^{(m)} \subseteq A \fr C_{m-1} \fr C_{m-2}
\fr \cdots \fr C_1
\]
and
\[
\Sigma_n^{(m)} \subseteq C_{m} \fr C_{m-1} \fr \cdots \fr C_1
\]
by
\begin{equation}\label{Eq_5x23_XinDfn}
\begin{split}
& \Xi_n^{(m)}
= \Bigl\{
q (b_1, b_2, \ldots, b_{k (n)} ) \mid
\\
& \hspace*{4em}
{\mbox{$q \in Q_n$ and $b_1, b_2, \ldots, b_{k (n)}
\in \io_m^{(m)} (E_n)
\cup \bigcup_{j = 1}^{m-1} \io_j^{(m)} (F_n^{(j)})$}}
\Bigr\}
\end{split}
\end{equation}
and
\begin{equation}\label{Eq_5x23_SmnDfn}
\begin{split}
& \Sm_n^{(m)}
= \Bigl\{
q (b_1, b_2, \ldots, b_{k (n)} ) \mid
\\
& \hspace*{6em}
{\mbox{$q \in Q_n$ and $b_1, b_2, \ldots, b_{k (n)}
\in \bigcup_{j = 1}^{m} \lambda_j^{(m)} (F_n^{(j)})$}}
\Bigr\}.
\end{split}
\end{equation}
Further set
\begin{equation}\label{Eq_5Z19_nn}
\Xi_n = \Xi_n^{(n)} \andeqn \Sm_n = \Sm_n^{(n)}.
\end{equation}

For $m \in \N$ it will follow
from (\ref{Eq_5Y24_EnFnIn}), (\ref{Eq_6524_33b}),
and~(\ref{Eq_5X26_ph}) that
\begin{equation}\label{Eq_5Y24_Xi_incr}
\Xi_m^{(m)} \subseteq \Xi_{m + 1}^{(m)} \subseteq
\Xi_{m + 2}^{(m)} \subseteq \cdots \subseteq A \fr C_{m-1} \fr C_{m-2}
\fr \cdots \fr C_1
\end{equation}
and from~(\ref{Eq_5X26_Incr}) and (\ref{Eq_6524_34B}) that
\begin{equation}\label{Eq_5Y24_Xi_mn}
{\overline{\bigcup_{n = m}^{\I} \Xi_n^{(m)}}} = A \fr  C_{m-1} \fr
C_{m-2} \fr \cdots \fr C_1 .
\end{equation}
Similarly,
\begin{equation}\label{Eq_5Y24_Sm_incr}
\Sm_m^{(m)} \subseteq \Sm_{m + 1}^{(m)} \subseteq
\Sm_{m + 2}^{(m)} \subseteq \cdots \subseteq  C_{m} \fr C_{m-1} \fr
\cdots \fr C_1
\end{equation}
and it follows from~(\ref{Eq_5x23_SmnDfn}), (\ref{Eq_6524_34B}),
and the first part of~(\ref{Eq_5X26_Incr}) that
\begin{equation}\label{Eq_5Y24_Sm_mn}
{\overline{\bigcup_{n = m}^{\I} \Sm_n^{(m)}}} =  C_{m} \fr C_{m-1} \fr
\cdots \fr C_1.
\end{equation}
Moreover, recalling (\ref{Eq_6524_idn}) for the definition
of $\id_{n - 1}$,
the relation~(\ref{Eq_5X26_Q_inc}) implies that
\[
(\ph_n \fr \id_{n - 1})
\left( \io_n^{(n)} (E_n) \cup
\bigcup_{j = 1}^{n-1} \io_j^{(n)} (F_n^{(j)}) \right)
\SQ \bigcup_{j = 1}^{n} \lambda_j^{(n)} (F_n^{(j)}),
\]
so (\ref{Eq_5x23_XinDfn}) and (\ref{Eq_5x23_SmnDfn}) imply
\begin{equation}\label{Eq_5X26_3star}
(\ph_n \fr \id_{n - 1}) (\Xi_n) \subseteq \Sm_n.
\end{equation}
Also, recalling (\ref{Eq_6524_iotan}) for the definition
of $\io_{n - 1}$ and using~(\ref{Eq_6524_33b}),
\[
\io_{n-1}
\left( \bigcup_{j = 1}^{n - 1} \lambda_j^{(n - 1)} (F_{n - 1}^{(j)}) \right)
= \bigcup_{j = 1}^{n-1} \io_j^{(n)} (F_{n - 1}^{(j)} )
\SQ \io_n^{(n)} (E_n) \cup \bigcup_{j = 1}^{n-1} \io_j^{(n)}(F_n^{(j)})
\]
and
\[
\beta_{n,n-1}
\left( \bigcup_{j = 1}^{n - 1} \ld_j^{(n - 1)} (F_{n - 1}^{(j)}) \right)
= \bigcup_{j = 1}^{n-1} \ld_j^{(n)} (F_{n - 1}^{(j)})
\SQ \bigcup_{j = 1}^{n} \ld_j^{(n)} (F_n^{(j)}),
\]
whence, by~(\ref{Eq_5X26_ph}),
\begin{equation}\label{Eq_5X26__4st}
\io_{n - 1} ( \Sm_{n - 1}) \subseteq \Xi_{n}
\end{equation}
and
\begin{equation}\label{Eq_5Y28_LdI}
\beta_{n,n-1} (\Sm_{n - 1}) \SQ \Sm_n.
\end{equation}

Since the elements of $\Xi_n$ are polynomials
in the images of the sets $E_n$ and $F_n^{(j)}$
in the free factors of $A \fr C_{n - 1} \fr C_{n-2} \fr \cdots \fr C_1$,
there will necessarily be $\delta_n > 0$
such that for any C*-algebra $B$, the following holds:
\begin{enumerate}
\setcounter{enumi}{\value{TmpEnumi}}
\item\label{5X23_A}
If $\alpha, \beta \colon A \fr  C_{n - 1} \fr C_{n-2} \fr \cdots \fr
C_1 \to B$
are homomorphisms which satisfy
\[
\bigl\| \alpha \bigl( \iota_n^{(n)} (a) \bigr)
- \beta \bigl( \iota_n^{(n)} (a) \bigr) \bigr\|
< \delta_n
\]
for any $a \in E_n$ and
\[
\bigl\| \alpha \bigl( \iota_k^{(n)} (c) \bigr)
- \beta \bigl( \iota_k^{(n)} (c) \bigr) \bigr\|
< \delta_n
\]
for any $c \in F_n^{(k)}$ and any $k \in \{1,2, 3, \ldots, n-1\}$, then
\[
\| \alpha (x) - \beta (x) \| < \eps_n
\]
for any $x \in \Xi_n$.
\setcounter{TmpEnumi}{\value{enumi}}
\end{enumerate}

Let $a_1, a_2, a_3, \ldots$ be a dense sequence in~$A$
and for $j \in \N$ let $c_1^{(j)}, c_2^{(j)}, c_3^{(j)}, \ldots$
be a dense sequence in~$C_j$.
Let $p_1, p_2, p_3, \ldots$ be a sequence
in $\Q [i] \langle t_1, t_2, \ldots \rangle$ such that for every $n \in \N$
we have $p_n \in \Q [i] \bigl\langle t_1, t_2, \ldots, t_n \bigr\rangle$,
and such that
\[
\{ p_1, p_2, p_3, \ldots \} = \Q [i] \langle t_1, t_2, \ldots \rangle.
\]

For $n, j \in \N$ with $n < j$, we set $F_n^{(j)} =
\{c_1^{(j)},c_2^{(j)},\ldots,c_n^{(j)}\}$.
When $n \geq j$, those sets will
be defined recursively below.

We now choose $k (1)$, $E_1$, $F_1^{(1)}$, and $Q_1$.
The sets $\Xi_1$ and $\Sm_1$ will then be given by (\ref{Eq_5x23_XinDfn})
and~(\ref{Eq_5x23_SmnDfn}),
we will have $\Xi_1^{(1)} = \Xi_1$, and $\dt_1$
will be chosen to satisfy (\ref{5X23_A}).
We will then verify the hypotheses (\ref{5X23_phnA}),
(\ref{5X23_phn_2}), (\ref{5X23_psinB}), and~(\ref{5X23_psin_2})
of \cref{thm_Elliott_intertwining} for $n = 1$.
Set $k (1) = 1$.
With $t_1$ denoting the monomial
$t_1 \in \Q [i] \langle t_1, t_2, \ldots \rangle$,
set
\[
Q_1 = \{ t_1, p_1 \},
\qquad
E_1 = \{ a_1 \},
\andeqn
F_1^{(1)} = \{ c_1^{(1)}, \varphi_1 (a_1) \}.
\]
The relations (\ref{Eq_5X27_kngn}) and~(\ref{Eq_5X26_Q_inc})
hold by construction, and (\ref{Eq_5X26_ph}) is vacuous.

Now define $\Xi_1 \subseteq A$ by~(\ref{Eq_5x23_XinDfn})
and $\Sm_1 \subseteq C_1$ by~(\ref{Eq_5x23_SmnDfn}).
Set $\Xi_1^{(1)} = \Xi_1$.
Choose $\dt_1 > 0$ such that (\ref{5X23_A}) holds.
Condition~(\ref{5X23_phn_2}) of
\cref{thm_Elliott_intertwining} for $n = 1$
follows from~(\ref{Eq_5X26_3star}).
Conditions (\ref{5X23_phnA}), (\ref{5X23_psinB}), and~(\ref{5X23_psin_2})
of \cref{thm_Elliott_intertwining} are vacuous for $n = 1$.

We next choose~$\widetilde{\alpha}_{2,1}$
and verify conditions (\ref{5X23_Acomm}) and~(\ref{5X23_SmD})
of \cref{thm_Elliott_intertwining} for $n = 1$.
Use~(\ref{I_5Y17_aue}) to choose a unitary $u_1 \in A \fr C_1$
such that for all $a \in E_1$, we have
\begin{equation}\label{Eq_5Y24_aue_1}
\bigl\| u_1 \bigl( \chi_2^{(1)} \circ \varphi_1 \bigr) (a) u_1^*
- \chi_1^{(1)} (a) \bigr\| < \dt_1.
\end{equation}
Set $\widetilde{\alpha}_{2,1} = \Ad (u_1^*) \circ \chi_1^{(1)}$.
Using~(\ref{5X23_A}) for $n = 1$ and~(\ref{Eq_5Y24_aue_1}),
for all $x \in \Xi_1$ we have
\begin{equation}\label{Eq_5X27_iou2}
\bigl\| \widetilde{\alpha}_{2,1} (x)
- \bigl( \chi_2^{(1)}  \circ \ph_1 \bigr) (x) \bigr\| < \ep_1.
\end{equation}
This is
\cref{thm_Elliott_intertwining}(\ref{5X23_Acomm}) for $n = 1$.
\cref{thm_Elliott_intertwining}(\ref{5X23_SmD}) holds for $n = 1$
because in fact $(\ph_2 \fr \id_{C_1}) \circ \io_1 = \beta_{2,1}$.

We proceed in a similar manner,
but first spell out the next stage for concreteness,
as extra complications occur at this point.
We construct $k (2)$, $E_2$, $F_2^{(1)}$, $F_2^{(2)}$, and $Q_2$.
We will then get $\Xi_2^{(1)}$, $\Xi_2^{(2)}$, $\Xi_2 = \Xi_2^{(2)}$,
and $\Sm_2$ from (\ref{Eq_5x23_XinDfn}), (\ref{Eq_5x23_SmnDfn}),
and~(\ref{Eq_5Z19_nn}),
and $\dt_2$ from (\ref{5X23_A}).

Choose $k_0 (2) \geq \max (k (1), 2)$ such that for every
$x \in \Xi_1 = \Xi_1^{(1)}$ there are
\[
q_x \in \Q [i] \langle t_1, t_2, \ldots, t_{2 k_0 (2)} \rangle,
\]
\[
b_{1, x}^{(1)}, b_{2, x}^{(1)}, \ldots, b_{k_0 (2), x}^{(1)} \in A,
\sandeqn
b_{1, x}^{(2)}, b_{2, x}^{(2)}, \ldots, b_{k_0 (2), x}^{(2)} \in C_1
\]
such that
\begin{equation}\label{Eq_5X27_qxD}
\Bigl\|
q_x \bigl( \bigl(
\chi_1^{(1)} \bigl( b_{j, x}^{(1)} \bigr) \bigr)_{1 \leq j \leq k_0
(2)},
\, \bigl( \chi_2^{(1)} \bigl( b_{j, x}^{(2)}
\bigr) \bigr)_{1 \leq j \leq k_0 (2)} \bigr)
- \widetilde{\alpha}_{2,1} (x) \Bigr\|
< \ep_2.
\end{equation}
Set $k (2) = 2 k_0 (2)$,
\begin{equation}\label{Eq_5X27_Q2D}
Q_2 = Q_1 \cup \{ p_2 \} \cup \{ q_x \mid x \in \Xi_1 \},
\end{equation}
\begin{equation}\label{Eq_5X27_E2Dfn}
E_2 = E_1 \cup \{ a_2 \}
\cup \bigl\{ b_{j, x}^{(1)} \mid
{\mbox{$j = 1, 2, \ldots, k_0 (2)$ and $x \in \Xi_1$}} \bigr\},
\end{equation}
\begin{equation}\label{Eq_F2_1_defn}
F_2^{(1)} = F_1^{(1)} \cup \{c_2^{(1)}\} \cup \ph_1 (E_2)
\cup \bigl\{ b_{j, x}^{(2)} \mid
{\mbox{$j = 1, 2, \ldots, k_0 (2)$ and $x \in \Xi_1$}} \bigr\},
\end{equation}
and
\begin{equation}\label{Eq_F2_2_defn}
F_2^{(2)} =  F_1^{(2)} \cup \{c_2^{(2)}\} \cup \ph_2 (E_2) .
\end{equation}
These choices give (\ref{Eq_5X27_kngn}), (\ref{Eq_5X26_Q_inc}),
and~(\ref{Eq_5X26_ph}) for $n = 2$.
Then define $\Xi_2^{(1)} \SQ A$ and $\Xi_2^{(2)} \SQ A \fr C_1$
by~(\ref{Eq_5x23_XinDfn}), set $\Xi_2 = \Xi_2^{(2)}$
by~(\ref{Eq_5Z19_nn}),
and define $\Sm_2 \SQ C_2 \fr C_1$ by~(\ref{Eq_5x23_SmnDfn}).
Choose $\dt_2 > 0$ following (\ref{5X23_A}).

We now verify the hypotheses
(\ref{5X23_phnA})--(\ref{5X23_psin_2}) of \cref{thm_Elliott_intertwining}
for $n = 2$.
They call for various distances to be less than~$\ep_2$.
In three of the four cases, the distance will actually be zero.
For \cref{thm_Elliott_intertwining}(\ref{5X23_phnA}), let $x \in \Xi_1$.
The element
\[
{\widetilde{x}}
= q_x \bigl( \bigl(
\chi_1^{(1)} \bigl( b_{j, x}^{(1)} \bigr) \bigr)_{1 \leq j \leq k_0 (2)},
\, \bigl( \chi_2^{(1)} \bigl( b_{j, x}^{(2)}
\bigr) \bigr)_{1 \leq j \leq k_0 (2)} \bigr)
\]
is in $\Xi_2$ by (\ref{Eq_5x23_XinDfn}), (\ref{Eq_5Z19_nn}),
(\ref{Eq_5X27_Q2D}), (\ref{Eq_5X27_E2Dfn}), and~(\ref{Eq_F2_1_defn}),
and satisfies
\[
\bigl\| \widetilde{\alpha}_{2,1} (x) - {\widetilde{x}} \bigr\| < \ep_2
\]
by~(\ref{Eq_5X27_qxD}).
For \cref{thm_Elliott_intertwining}(\ref{5X23_phn_2}), let $x \in \Xi_2$.
We have
\[
(\ph_2 \fr \id_{C_1}) (\Xi_2) \subseteq \Sm_2
\]
by~(\ref{Eq_5X26_3star}).
Therefore in fact $y = (\ph_2 \fr \id_{C_1}) (x) \in \Sm_2$.
For \cref{thm_Elliott_intertwining}(\ref{5X23_psinB}), let $y \in \Sm_1$.
Set ${\widetilde{y}} = \beta_{2,1} (y)$.
Then in fact ${\widetilde{y}} \in \Sm_2$ by~(\ref{Eq_5Y28_LdI}).
For \cref{thm_Elliott_intertwining}(\ref{5X23_psin_2}), let $y \in \Sm_1$.
Set $x = \io_1 (y)$.
Then in fact $x \in \Xi_2$ by~(\ref{Eq_5X26__4st}).

%%%%%%%%%%%%% 20.5 EDITED UNTIL HERE

Now use~(\ref{I_5Y17_aue}) to choose a unitary $u_2 \in A \fr C_2$
such that for any $a \in E_2$ we have
\[
\bigl\| u_2 \bigl( \chi_2^{(2)} \circ \varphi_2 \bigr) (a) u_2^* -
\chi_1^{(2)} (a) \bigr\|
< \delta_2.
\]
Observe now that
\[
\chi_1^{(2)} \fr \id_{C_1} = \alpha_{3, 2}.
\]
Define $\widetilde{\alpha}_{3, 2} \colon A \fr C_1 \to A \fr C_2 \fr C_1$ by
\[
\widetilde{\alpha}_{3, 2}
= ( \Ad (u_2^*) \fr \id_{C_1} ) \circ \alpha_{3, 2}.
\]
Let $x \in \Xi_2$.
Since $\delta_2$ satisfies~(\ref{5X23_A}),
it follows that
\[
\left\| \bigl[ \bigl(  \chi_2^{(2)} \circ \varphi_2 \bigr)
\fr \id_{C_1} \bigr] (x)
-\widetilde{\alpha}_{3, 2} (x) \right\| < \eps_2.
\]
Since $\chi_2^{(2)} \fr \id_{C_1} = \io_{2}$, we thus have
\[
\| (\io_{2} \circ [\ph_2 \fr \id_{C_1}]) (x)
- \widetilde{\alpha}_{3, 2} (x) \|
< \ep_2.
\]
This last statement is condition~(\ref{5X23_Acomm})
of \cref{thm_Elliott_intertwining}.
Condition~(\ref{5X23_SmD}) is immediate by~(\ref{Eq_6524_lcomm}).

We now give the general construction for $n \geq 3$.
One further complication occurs here,
involving the sets $\Xi_{n - 1}^{(m)}$ for $m < n - 1$.
Suppose $k (n - 1)$, $E_{n - 1}$, $F_{n - 1}^{(j)}$
for $j = 1, 2, \ldots, n - 1$, $Q_{n - 1}$,
and ${\widetilde{\alpha}}_{n,n-1}$ have been chosen,
and hence also $\Xi_{n - 1}$, $\Xi_{n - 1}^{(m)}$ for $m \leq n - 1$,
$\Sm_{n - 1}$, and $\Sm_{n - 1}^{(m)}$ for $m \leq n - 1$.

Choose $k_0 (n) \geq \max (k (n - 1), n)$ such that for every
$m \in \{ 1, 2, \ldots, n - 1 \}$ and $x \in \Xi_{n - 1}^{(m)}$
there are
\[
q_x^{(m)} \in \Q [i] \langle t_1, t_2, \ldots, t_{n k_0 (n)} \rangle,
\quad
b_{1, x}^{(n, m)}, b_{2, x}^{(n, m)}, \ldots, b_{k_0 (n), x}^{(n, m)}
\in A,
\]
and for $j = 1, 2, \ldots, n - 1$ we have
\[
b_{1, x}^{(j, m)}, b_{2, x}^{(j, m)}, \ldots, b_{k_0 (n), x}^{(j, m)}
\in C_j
\]
such that
\begin{equation}\label{Eq_5Y24_qxD}
\Bigl\|
q_x^{(m)} \bigl( \bigl(
\io_n^{(n)} \bigl( b_{s, x}^{(n, m)} \bigr) \bigr)_{1 \leq s \leq k_0
(n)},
\, \ldots, \, \bigl( \io_1^{(n)} \bigl( b_{s, x}^{(1, m)}
\bigr) \bigr)_{1 \leq s \leq k_0 (n)} \bigr)
- \widetilde{\af}_{n, m} (x) \Bigr\|
< \ep_n.
\end{equation}

Set $k (n) = n k_0 (n)$,
\begin{equation}\label{Eq_5Y24_Q2D}
Q_n
= Q_{n - 1}
\cup \{ p_n \}
\cup
\bigl\{ q_x^{(m)} \mid {\mbox{$1 \leq m \leq n - 1$
and $x \in \Xi_{n - 1}^{(m)}$}} \bigr\},
\end{equation}
\begin{equation}\label{Eq_5Y24_E2Dfn}
\begin{split}
E_n
& = E_{n - 1} \cup \{ a_n \}
\\
& \hspace*{1em} {\mbox{}}
\cup \bigl\{ b_{s, x}^{(n, m)} \mid
{\mbox{$1 \leq s \leq k_0 (n)$, $1 \leq m \leq n - 1$,
and $x \in \Xi_{n - 1}^{(m)}$}} \bigr\},
\end{split}
\end{equation}
and for $j = 1, 2, \ldots, n - 1$,
\begin{equation}\label{Eq_5Y24_F2_defn}
\begin{split}
F_n^{(j)}
& = F_{n - 1}^{(j)} \cup \{ c_n^{(j)} \} \cup \ph_j (E_n)
\\
& \hspace*{1em} {\mbox{}}
\cup \bigl\{ b_{s, x}^{(j, m)} \mid
{\mbox{$1 \leq s \leq k_0 (n)$, $1 \leq m \leq n - 1$,
and $x \in \Xi_{n - 1}^{(m)}$}} \bigr\},
\end{split}
\end{equation}
and
\[
F_n^{(n)} = F_{n - 1}^{(n)} \cup \{ c_n^{(n)} \} \cup \ph_n (E_n).
\]
These choices give (\ref{Eq_5X27_kngn}), (\ref{Eq_5X26_Q_inc}),
and~(\ref{Eq_5X26_ph}) for~$n$.
For $m = 1, 2, \ldots, n$ define
\[
\Xi_{n}^{(m)} \SQ A \fr C_{m - 1} \fr C_{m-2} \fr \cdots \fr C_1
\]
by~(\ref{Eq_5x23_XinDfn}) and
\[
\Sm_{n}^{(m)} \SQ C_m \fr C_{m - 1} \fr \cdots \fr C_1
\]
by~(\ref{Eq_5x23_SmnDfn}).
Then define
\[
\Xi_n \SQ A \fr C_{n - 1} \fr C_{n-2} \fr \cdots \fr C_1
\quad {\mbox{and}} \quad
\Sm_n \SQ  C_n \fr C_{n - 1} \fr \cdots \fr C_1
\]
by~(\ref{Eq_5Z19_nn}).
Choose $\dt_n > 0$ following (\ref{5X23_A}).

We check the hypotheses
(\ref{5X23_phnA})--(\ref{5X23_psin_2}) of \cref{thm_Elliott_intertwining}
for~$n$.
For~(\ref{5X23_phnA}), let $x \in \Xi_{n - 1}$.
Then
\[
\widetilde{x}
= q_x^{(n - 1)} \bigl( \bigl(
\io_n^{(n)}
\bigl( b_{s, x}^{(n, \, n - 1)} \bigr) \bigr)_{1 \leq s \leq k_0 (n)},
\, \ldots, \, \bigl( \io_1^{(n)} \bigl( b_{s, x}^{(1, \, n - 1)}
\bigr) \bigr)_{1 \leq s \leq k_0 (n)} \bigr)
\]
satisfies
$\| \widetilde{\alpha}_{n, n - 1} (x) - \widetilde{x} \|
< \ep_n$.
Also, for $s = 1, 2, \ldots, k_0 (n)$,
the elements $b_{s, x}^{(n, \, n - 1)}$ are in $E_n$
by~(\ref{Eq_5Y24_E2Dfn}),
the elements $b_{s, x}^{(j, \, n - 1)}$, for $j = 1, 2, \ldots, n-1$,
are in $F_n^{(j)}$ by~(\ref{Eq_5Y24_F2_defn}),
and $q_x^{(n - 1)} \in Q_n$ by~(\ref{Eq_5Y24_Q2D}).
Therefore $\widetilde{x} \in \Xi_n$ by~(\ref{Eq_5x23_XinDfn}).
Conditions (\ref{5X23_phn_2}), (\ref{5X23_psinB}), and~(\ref{5X23_psin_2})
follow from $(\ph_n \fr \id_{n-1} ) (\Xi_n) \subseteq \Sm_n$,
which is~(\ref{Eq_5X26_3star}),
$\beta_{n,n-1} (\Sm_{n - 1}) \SQ \Sm_n$,
which is~(\ref{Eq_5Y28_LdI}),
and $\io_{n - 1} ( \Sm_{n - 1}) \subseteq \Xi_{n}$,
which is~(\ref{Eq_5X26__4st}).

We claim that the following also holds:
\begin{enumerate}
\setcounter{enumi}{\value{TmpEnumi}}
\item\label{5X28_Item4}
Whenever $m, n \in \N$ satisfy $m \leq n - 1$,
and for every $x \in \Xi_{n - 1}^{(m)}$,
there is $\widetilde{x} \in \Xi_n$ such that
$\| \widetilde{x} - \widetilde{\af}_{n, m } (x) \| < \ep_n$.
\end{enumerate}
(For $n = 2$, there was only one case, $m = 1$, which was just the
verification of \cref{thm_Elliott_intertwining}(\ref{5X23_phnA}).)
To prove the claim, take $\widetilde{x}$ to be given by
\[
\widetilde{x}
= q_x^{(m)} \bigl( \bigl(
\io_n^{(n)} \bigl( b_{s, x}^{(n, m)} \bigr) \bigr)_{1 \leq s \leq k_0 (n)},
\, \ldots, \, \bigl( \io_1^{(n)} \bigl( b_{s, x}^{(1, m)}
\bigr) \bigr)_{1 \leq s \leq k_0 (n)} \bigr),
\]
the element used in~(\ref{Eq_5Y24_qxD}),
giving $\| \widetilde{x} - \widetilde{\af}_{n, m } (x) \| < \ep_n$.
We have $\widetilde{x} \in \Xi_n$ by
(\ref{Eq_5Y24_Q2D}), (\ref{Eq_5Y24_E2Dfn}),
(\ref{Eq_5Y24_F2_defn}), (\ref{Eq_5x23_XinDfn}), and~(\ref{Eq_5Z19_nn}).
This proves~(\ref{5X28_Item4}).

Next, use~(\ref{I_5Y17_aue}) to choose
$u_n \in A \fr C_n$ such that for any $a \in E_{n}$ we have
\[
\bigl\| u_n
\bigl( \chi_2^{(n)} \circ \varphi_n \bigr) (a) u_n^* -
\chi_1^{(n)} (a) \bigr\|
< \delta_n \, .
\]
Recalling~(\ref{Eq_6524_idn}) for the definition of $\id_{n - 1}$,
define
\[
\widetilde{\alpha}_{n + 1, n}
\colon A \fr C_{n-1} \fr C_{n-2} \fr \cdots \fr C_1 \to A \fr C_{n} \fr
C_{n-1} \fr \cdots \fr C_1
\]
by
\[
\widetilde{\alpha}_{n + 1, n}
 = ( \Ad (u_n^*) \fr \id_{n - 1} ) \circ \alpha_{n + 1, n}.
\]
For convenience of notation, define
\[
\om_n \colon  A \fr C_{n-1} \fr C_{n-2} \fr \cdots \fr C_1
  \to A \fr C_{n} \fr C_{n-1} \fr \cdots \fr C_1
\]
by
\begin{equation}\label{Eq_6124_Omn}
\om_n = \bigl( \chi_2^{(n)} \circ \varphi_n \bigr) \fr \id_{n - 1}.
\end{equation}
A computation shows that for $j = 1, 2, \ldots, n-1$ and $c \in C_j$, we
have
\[
\bigl( \om_n \circ \io_j^{(n)} \bigr) (c)
= \io_{j}^{(n + 1)} (c)
= \bigl( \widetilde{\alpha}_{n+1,n} \circ \io_j^{(n)} \bigr) (c).
\]
Also, for $a \in A$, the element
$\om_n \bigl( \io_n^{(n)}(a) \bigr)$ is
$\chi_2^{(n)} (\ph_n (a))$, viewed inside the first two free factors
$A \fr C_n$ of $A \fr C_n \fr C_{n-1} \fr \cdots \fr C_1$.
Thus, if $a \in E_n$, we have
\[
\bigl\| \om_n \bigl( \io_n^{(n)} (a) \bigr)
- \widetilde{\alpha}_{n+1, n}
\bigl( \io_n^{(n)} (a) \bigr) \bigr\| < \dt_n.
\]
Since $\delta_{n}$ satisfies~(\ref{5X23_A}),
and recalling~(\ref{Eq_6124_Omn}), for all $x \in \Xi_n$ we have
\[
\bigl\| \bigl[  \iota_n \circ \bigl(\varphi_n
\fr \id_{n-1} \bigr) \bigr] (x)
- \widetilde{\alpha}_{n+1, n} (x) \bigr\|
< \eps_{n}.
\]
We have proved condition
(\ref{5X23_Acomm}) of \cref{thm_Elliott_intertwining}.
As before, \cref{thm_Elliott_intertwining}(\ref{5X23_SmD})
is immediate from~(\ref{Eq_6524_lcomm}).

This completes the construction of the modified version of the diagram.
It remains to prove conditions
(\ref{5X23_Xin}) and~(\ref{5X23_Smn}) of \cref{thm_Elliott_intertwining}.

For \cref{thm_Elliott_intertwining}(\ref{5X23_Xin}),
let $m \in \N$, let $x \in A \fr C_{m-1} \fr C_{m-2} \fr \cdots \fr C_1 $,
and let $\ep > 0$.
By (\ref{Eq_5Y24_Xi_incr}) and~(\ref{Eq_5Y24_Xi_mn}),
there are $n_0 \geq m$ and $x_0 \in \Xi_{n_0}^{(m)}$ such that
$\ep_{n_0} \leq \frac{\ep}{2}$ and $\| x_0 - x \| < \frac{\ep}{2}$.
Let $n > n_0$.
Use~(\ref{5X28_Item4}) to choose $\widetilde{x} \in \Xi_n$
such that $\| \widetilde{x} - \widetilde{\af}_{n, m } (x_0) \| < \ep_n$.
Then
\[
\| \widetilde{x} - \widetilde{\af}_{n, m } (x) \|
\leq \| \widetilde{x} - \widetilde{\af}_{n, m } (x_0) \| + \| x_0 - x \|
< \ep_n + \frac{\ep}{2}
\leq \ep,
\]
as desired.

For \cref{thm_Elliott_intertwining}(\ref{5X23_Smn}),
let $m \in \N$, let $y \in C_{m} \fr C_{m-1} \fr \cdots \fr C_1 $, and
let $\ep > 0$.
Since for any $j \in \N$ we have $F_1^{(j)} \SQ F_2^{(j)} \SQ \cdots$ and
$\bigcup_{n = 1}^{\I} F_n^{(j)}$ is
dense in~$C_j$ (by (\ref{Eq_6524_33b}) and~(\ref{Eq_6524_34B})),
there are $r, n_0 \in \N$,
$q \in \Q [i] \langle t_1, t_2, \ldots, t_{r m} \rangle$,
and $b_l^{(j)} \in F_{n_0}^{(j)}$
for $l = 1, 2, \ldots, r$ and $j = 1, 2, \ldots, m$,
such that the element
\[
y_0
= q \bigl( \bigl(
\ld_m^{(m)} \bigl( b_{l}^{(m)} \bigr) \bigr)_{1 \leq l \leq r},
\, \ldots, \, \bigl( \ld_1^{(m)} \bigl( b_{l}^{(1)}
\bigr) \bigr)_{1 \leq l \leq r} \bigr)
\]
satisfies $\| y_0 - y \| < \ep$.
Choose $n \in \N$ so large that
\[
n \geq \max (m, n_0, rm) \andeqn q \in \{ p_1, p_2, \ldots, p_n \}.
\]
Then $k (n) \geq n \geq rm$ and $q \in Q_n$.
It now follows from~(\ref{Eq_5x23_SmnDfn}) and $b_l^{(j)} \in F_{n}^{(j)}$
that $\bt_{n, m} (y_0) \in \Sm_n$.
Since $\| y_0 - y \| < \ep$, we have
$\| \bt_{n, m} (y_0) - \bt_{n, m} (y) \| < \ep$.
This proves~(\ref{5X23_Smn}).

By \cref{thm_Elliott_intertwining}, we have
\[
\begin{split}
&
\dirlim
 \bigl( ( A \fr C_{n-1} \fr C_{n-2} \fr \cdots \fr C_1 )_{n \in \N},
 ( \widetilde{\alpha}_{n + 1, \, n})_{n \in \N} \bigr)
\\
& \hspace*{3em} {\mbox{}}
 \cong \dirlim
 \bigl( ( C_n \fr C_{n-1} \fr \cdots \fr C_1)_{n \in \N},
    ( \beta_{n + 1, \, n})_{n \in \N} \bigr)  \,.
\end{split}
\]
It remains only to identify the left hand direct limit with the direct limit
using the original maps $\alpha_{n+1,n}$.
For $n \in \N$, set
\[
\gamma_{n+1} = \Ad (u_n^*) \fr \id_{n - 1}
\in \Aut (A \fr C_n \fr C_{n-1} \fr \cdots \fr C_1).
\]
Thus $\widetilde{\alpha}_{n+1,n}=\gamma_{n+1} \circ \alpha_{n+1,n}$.
For $m \geq k > 1$, let $\beta_m^{(k)}$ be the automorphism of
$A \fr C_{m-1} \fr \cdots \fr C_1$ which acts as
$\Ad (u_{k-1}^*)$ on the factors generated by $A$ and
$C_{k-1}$ and fixes all other free factors.
Then the hypotheses of \cref{lem_isomorphism_different_systems_0} are
satisfied, so
\[
\begin{split}
&
\dirlim
 \bigl( ( A \fr C_{n-1} \fr C_{n-2} \fr \cdots \fr C_1 )_{n \in \N},
 ( \widetilde{\alpha}_{n + 1, \, n})_{n \in \N} \bigr)
\\
& \hspace*{3em} {\mbox{}}
\cong
\dirlim
 \bigl( ( A \fr C_{n-1} \fr C_{n-2} \fr \cdots \fr C_1 )_{n \in \N},
 ( \alpha_{n + 1, \, n})_{n \in \N} \bigr).
\end{split}
\]
The right hand direct limit above is $A \fr \bigastr_{n=1}^{\infty} C_n$.
This concludes the proof.
\end{proof}

\begin{proof}[Proof of \cref{thm_main_RRZ}]
We describe the modifications
to the proof of \cref{thm_main} needed to replace $A$ with $C (X)$.
For each block $I_n$ as in the hypothesis,
define $D_n = \bigastr_{k \in I_n} C_k$.
Use $D_n$ in place of $C_n$.
By \cref{lemma_Infinite_Free_Product_Selfless},
the algebras $D_n$ and also $C (X) \fr D_n$
are simple, have unique tracial states, have stable rank one,
and have strict comparison.
Furthermore, there is a unital embedding $\zt_n \colon \Zh \to D_n$.
The algebra $C (X) \fr D_n$ has real rank zero by
\cref{C_2622_RanK0}(\ref{I_C_2622_RanK0_RRZ}).
Moreover, $C (X) \fr D_n$ is exact by
exactness of the free factors and \cite[Corollary~4.3]{DkmShk}.

By \cite[Corollary C]{abstract_classification},
there exists a tracial state preserving embedding
$\te \colon C (X) \to \Zh$.
Since $\Zh$ and $D_n$ have unique \tst{s},
$\zt_n$ is trace preserving.
Therefore so is $\zt_n \circ \te$.
We use \cite[Theorem 4.8]{Matui} for the class~$\cT'$.
This class is defined in \cite[Definition 2.2]{Matui},
and the properties proved in the previous paragraph
show that $C (X) \fr D_n \in \cT'$.
(This is where exactness and real rank zero are used.)
Theorem 4.8 of \cite{Matui} thus implies that
any two tracial state preserving unital homomorphisms
from $C (X)$ to $C (X) \fr D_n$ are approximately unitarily equivalent.
This fact is used to choose the unitaries $u_n$
in the proof of \cref{thm_main}.
The rest of the argument is the same.
\end{proof}

We give a few corollaries of \cref{thm_main} and \cref{thm_main_RRZ}.
Recall that we use Lebesgue measure on $[0, 1]$.

\begin{Cor}
We have $\Zh^{\fr \infty} \cong C ([0, 1])^{\fr \infty}$.
\end{Cor}

\begin{proof}
By two applications of \cref{thm_main} and \cref{Ex_6524_Const},
we have
$C ([0, 1]) \fr \Zh^{\fr \infty} \cong \Zh^{\fr \infty}$
and $\Zh \fr C ([0, 1])^{\fr \infty} \cong C ([0, 1])^{\fr \infty}$.
Therefore
\[
C ([0, 1])^{\fr \infty} \fr \Zh^{\fr \infty}
\cong [C ([0, 1]) \fr \Zh^{\fr \infty}]^{\fr \infty}
\cong [\Zh^{\fr \infty}]^{\fr \infty}
\cong \Zh^{\fr \infty},
\]
and similarly
$C ([0, 1])^{\fr \infty} \fr \Zh^{\fr \infty}
\cong C ([0, 1])^{\fr \infty}$.
\end{proof}

\begin{Cor}\label{C_5Z19_0CpP}
Let $X$ and $Y$ be \cms{s},
equipped with probability measures $\mu$ and $\nu$ with full support.
Suppose $X$ is contractible and $Y$ has a compact open subset
$T$ such that $\nu (T)$ is irrational.
Then
\[
C (X) \fr C (Y)^{\fr \infty} \cong C (Y)^{\fr \infty}.
\]
\end{Cor}

\begin{proof}
The algebra $C (Y)^{\fr \infty}$
is exact by \cite[Corollary~4.3]{DkmShk}.
Let $\ta$ be the \tst{} on $C (Y)$ induced by~$\nu$.
Since $\ta_* (K_0 (C (Y)))$ contains both $1$ and some irrational
number, it is dense in~$\R$, and also $\At (C (Y), \ta) \neq 1$.
So \cref{thm_main_RRZ} applies.
\end{proof}

\begin{exa}\label{E_5Z20_01Z}
Consider $\C \oplus \C$, endowed with any tracial state
$\sm$ such that $\sm ( 1, 0) \not\in \Q$.
For any $d \in \N \cup \{\infty\}$,
using Lebesgue measure on $[0, 1]^d$,
\cref{thm_main_RRZ} and \cref{Ex_6524_CplusC} imply
\[
C ([0, 1]^d) \fr (\C \oplus \C)^{\fr \infty}
\cong (\C \oplus \C)^{\fr \infty}.
\]
\end{exa}

\begin{exa}\label{E_5Z20_Cantor}
Let $X$ be a contractible \cms{} and let $Y$ be the Cantor set,
both equipped with probability measures with full support.
Then
\[
C (X) \fr C (Y)^{\fr \infty} \cong C (Y)^{\fr \infty}.
\]

To see this, simply observe that $Y$ must have compact open subsets
of arbitrarily small measure.
Since the measures of these sets are nonzero,
the range of the induced trace on $K_0 (C (Y))$ is dense in~$\R$.
Now follow the proof of \cref{C_5Z19_0CpP}.
\end{exa}

\begin{Cor}\label{C_5Z19_CtrDisc}
Let $X$ and $Y$ be \cms{s},
equipped with probability measures with full support,
and such that $X$ is contractible.
Let $D$ be a simple separable infinite dimensional unital exact \ca{}
with real rank zero, and let $\ta$ be a \tst{} on~$D$.
Use $\ta$ and the measure on $Y$
to get a \tst~$\sm$ on $C (Y, D)$ in the obvious way.
Then
\[
C (X) \fr C (Y, D)^{\fr \infty} \cong C (Y, D)^{\fr \infty}.
\]
\end{Cor}

\begin{proof}
The algebra $C (Y, D)$ is clearly exact.
The subgroup $\ta_* (K_0 (D))$ is dense in~$\R$,
so $\sm_{*} (K_0 (C (Y, D) ))$ is dense in~$\R$.
Also, $D$ has a Haar unitary $u$ by \cite[Corollary 5.6]{Thiel_diffuse},
and consideration of the constant function with value $u + u^*$
shows that $\At (C (Y, D),\sigma) = 0$.
So \cref{thm_main_RRZ} applies, using any partition of $\N$
into infinitely many infinite sets.
\end{proof}

\begin{exa}\label{E_5Z20_Mn}
Let $D$ be a UHF algebra, or an irrational rotation algebra.
For any $d \in \N \cup \{\infty\}$ we have
\[
C ([0, 1]^d) \fr D^{\fr \infty} \cong D^{\fr \infty}.
\]
\end{exa}

\section{Counterexamples and open questions}\label{Sec_5Z17_Open}

We begin with a discussion of limitations on extensions of
our results.
We then list a few natural open followup questions.

No K-theoretic classification can go very far into algebras
constructed from reduced free products.
An immediate reason is that K-theory does not detect $\Zh$-stability,
not even in the presence of real rank zero.

\begin{exa}\label{Ex_6601_UHF}
Let $D$ be the $2^{\I}$~UHF algebra.
Set
\[
A = D \otimes C ([0, 1])^{\fr \I}
\andeqn
B = D \fr C ([0, 1])^{\fr \I},
\]
using the unique \tst{} on~$D$ and, as usual, Lebesgue measure
on $[0, 1]$.
\cref{L_2622_KthGen} implies that $K_0 (C ([0, 1])^{\fr \I}) \cong \Z$,
generated by $[1]$, and $K_1 (C ([0, 1])^{\fr \I}) = 0$.
Therefore the K{\"u}nneth formula shows that
$K_0 (A) \cong \Z [\frac{1}{2}]$, with $1_A$ corresponding to
$1 \in \Z \SQ \Z [\frac{1}{2}]$, and $K_1 (A) = 0$.
Also, \cref{L_2622_KthGen} and stable finiteness imply that
$K_0 (B) \cong \Z [\frac{1}{2}]$, with $1_B$ corresponding to
$1 \in \Z \SQ \Z [\frac{1}{2}]$, and $K_1 (B) = 0$.
Since $A$ and $B$ have unique \tst{s}
(see \cite[Part~(3) of the corollary on page~431]{Avitzour}),
it follows that they have isomorphic Elliott invariants.
However, $A$ is $\Zh$-stable, even $D$-stable,
while $B$ is not $\Zh$-stable.
\end{exa}

However, much more dramatic phenomena occur.

\begin{exa}\label{Ex_6603_Opp}
In~\cite{PhlVls1}, with $D$ being the $3^{\I}$~UHF algebra,
there is an action $\af$ of $\Z / 3 \Z$ on
\[
A = D \otimes
   \bigl[ \C^3 \fr C ([0, 1]) \fr C ([0, 1]) \fr C ([0, 1]) \bigr]
   = D \otimes [ \C^3 \fr C ([0, 1])^{\fr 3} ]
\]
such that $C^* (\Z / 3 \Z, A, \af)$
is not isomorphic to its opposite algebra.
Thus,  $C^* (\Z / 3 \Z, A, \af)$ and its opposite can't be distinguished by 
K-theoretic
invariants, not even if one includes the Cuntz semigroup,
and common properties, such as $E$-stability for any strongly
selfabsorbing \ca~$E$, also do not help.
This crossed product is even $\Zh$-stable, in fact, $D$-stable,
and, by \cite[Proposition~4.11]{PhlVls2},
satisfies the Universal Coefficient Theorem.
\end{exa}

Finiteness of the reduced free product is probably not the issue here.
Although we have not yet checked,
the same proof likely applies to
$D \otimes [\C^3 \fr C ([0, 1])^{\fr 3}]^{\fr \infty}$,
in which the action on the first free factor
$\C^3 \fr C ([0, 1])^{\fr 3}$ in the infinite free product is $\af$
and the actions on the remaining copies of $\C^3 \fr C ([0, 1])^{\fr 3}$
are trivial.

Bad behavior of the action of $\Z / 3 \Z$ is also not the issue.
It is not known whether, assuming simplicity,
the crossed product of an Elliott classifiable \ca{}
by a finite group is again Elliott classifiable.
However, the action $\af$ above has the continuous Rokhlin property,
and the crossed product
of a simple separable nuclear unital $\Zh$-stable \ca{}
satisfying the Universal Coefficient Theorem
by an action of a finite group with this property
is again $\Zh$-stable and, by Proposition~3.8 of~\cite{PhlVls2},
satisfies the Universal Coefficient Theorem.

\cref{thm_main} and \cref{thm_main_RRZ}
fail if one places no restrictions on the sizes of atoms.
Our example uses the following proposition.

\begin{Prop}\label{P_6604_NonSimp}
Let $\lambda_1, \lambda_2, \ldots \in (0, \frac{1}{2}]$.
Take $C_n = \C \oplus \C$,
with the \tst{} $\sigma_n (a, b) = \lambda_n a + (1-\lambda_n) b$.
Assume that $\sum_{n = 1}^{\infty} \lambda_n \leq 1$.
Then there is a unital \hm{}
$\om \colon \bigastr_{n = 1}^{\infty} \, (C_n, \sm_n) \to \C$.
\end{Prop}

\begin{proof}
For $n \in \N$ let $p_n, q_n \in C_n$ be the \pj{s}
$p_n = (1, 0)$ and $q_n = (0, 1)$,
so that, in particular, $\sm_n (q_n) = 1 - \lambda_n$.
Further, for $n \in \N$ set $A_n = \bigastr_{k = 1}^{n} \, (C_k, \sm_k)$,
let $\ta_n$ be the reduced free product trace on~$A_n$,
let $H_n$ be the associated free product Hilbert space
(\cite[Definition~1.3.1]{VoiDyNi}),
and let $\pi_n \colon A_n \to L (H_n)$ be the free product representation
(\cite[Definition~1.5.1]{VoiDyNi}).

Let $n \in \N$.
We construct a \uhm{} $\om_n \colon A_n \to \C$
such that for $k = 1, 2, \ldots, n$ we have $\om_n (q_k) = 1$.
Set $b_n = \sum_{k = 1}^{n} q_k$.
By free independence,
the spectral distribution of $b_n$ with respect to $\ta_n$ is
given by free additive convolution
(\cite[Definition 3.1.1]{VoiDyNi}).
Since $\sum_{j = 1}^{k} (1 - \lambda_j) > k - 1$
for $k = 1, 2, \ldots, n$,
induction on $k$ and the atom formula for free additive
convolution, \cite[Theorem~7.4]{Bercovici_Voiculescu_Regularity}, imply
that $\{ n \}$ is an atom for the spectral measure of $b_n$
with respect to~$\ta_n$.
Since $\ta_n$ is a vector state on $\pi_n (A_n)$,
it follows (using bounded Borel functional calculus) that the \pj{}
$e_n = \ch_{ \{ n \} } (\pi_n (b_n))$ is nonzero.

We claim that $e_n$ commutes with $\pi_n (A_n)$.
To prove the claim, it suffices to show that
$e_n$ commutes with $\pi_n (q_k)$ for $k = 1, 2, \ldots, n$.
It is enough to show that $\pi_n (q_k) e_n = e_n$.
Let $\xi \in e_n H_n$; we show that $\pi_n (q_k) \xi = \xi$.
\Wolog{} $\| \xi \| = 1$.
We have $\| \pi_n (q_k) \xi \| \leq \| q_k \| = 1$
for $k = 1, 2, \ldots, n$, while $\pi_n (b_n) \xi = n \xi$, so
\[
\left\| \sum_{k = 1}^{n} \pi_n (q_k) \xi \right\|
 = \| \pi_n (b_n) \xi \| = n.
\]
Therefore $\pi_n (q_k) \xi = \xi$ for $k = 1, 2, \ldots, n$.
The claim is proved.

It follows from the claim that there is a \uhm{}
$\om_n \colon A_n \to \C$ such that for all $a \in A_n$ we have
$e_n \pi_n (a) e_n = \om_n (a) e_n$.
For $k = 1, 2, \ldots, n$, since $\pi_n (q_k) e_n = e_n$,
we get $\om_n (q_k) = 1$.
This completes the construction.

Identifying $A_n$ with its image
in $\bigastr_{n = 1}^{\infty} \, (C_n, \sm_n)$, we get
\[
\bigastr_{n = 1}^{\infty} \, (C_n, \sm_n)
 = {\overline{ \bigcup_{n = 1}^{\I} A_n }}.
\]
Clearly $\om_{n + 1} |_{A_n} = \om_n$ for $n \in \N$.
The existence of $\om$ follows.
\end{proof}

\begin{exa}\label{Ex_6604_NonSmp}
Let the notation be as in \cref{P_6604_NonSimp},
with the choices $\ld_1 = \frac{1}{2}$ and
$\ld_n = 1 / 3^n$ for $n = 2, 3, \ldots$.
Then \cref{P_6604_NonSimp} implies that the
resulting reduced free product
$C = \bigastr_{n = 1}^{\infty} \, (C_n, \sm_n)$ is not simple.
However, $C_1$ contains a unitary with trace zero,
so $C$ does also.
Therefore \cite[Part~(3) of the corollary on page~431]{Avitzour}
implies that $C ([0, 1]) \fr C$ is simple.
% $C ([0, 1]) \fr \bigastr_{n = 1}^{\infty} \, (C_n, \sm_n)$ is simple.
\end{exa}

In fact, whenever $\sum_{n = 1}^{\I} \ld_n \leq 1$,
the algebra $C$ in \cref{Ex_6604_NonSmp}
does not freely absorb $C ([0, 1])$.
Otherwise, we would have
$C ([0, 1]) \fr [C ([0, 1]) \fr C] \cong C$.
Now $C$ is not simple by \cref{P_6604_NonSimp},
but $C ([0, 1]) \fr [C ([0, 1]) \fr C] \cong C$ is simple
by \cite[Part~(3) of the corollary on page~431]{Avitzour},
regardless of what $C$ is.

The first followup question
is a special case of \cite[Problem XCVIII]{99problems}.

\begin{qst}\label{Q_5Z19_Iso}
Using Lebesgue measure on $[0, 1]$,
do we have
\[
C ([0, 1]) \fr C ([0, 1]) \cong \Zh \fr C ([0, 1]) \cong \Zh \fr \Zh?
\]
\end{qst}

The recent result of \cite{hayes_ke_robert},
which shows that the C*-algebras in the question are selfless,
suggests a possible strategy to push our results
from the case of infinite free products to finite free products,
but this does not appear to be straightforward.
We also seem to make no progress towards deciding whether
$C ([0, 1])^{\fr k} \cong C ([0, 1])^{\fr l}$ when $k \neq l$,
another special case of \cite[Problem XCVIII]{99problems}.

\begin{qst}\label{Q_5Z19_01n}
For $n > 1$, consider $C ([0, 1]^n)$ with Lebesgue measure.
Do we have
$C ([0, 1]^n) \fr C ([0, 1])^{\fr \infty} \cong C ([0, 1])^{\fr \infty}$?
\end{qst}

The following related question is perhaps simpler.
It is one of the simplest cases of \cref{E_5Z20_01Z} without
real rank zero.

\begin{qst}\label{Qst_6129_half}
Consider $\C \oplus \C$, endowed with the tracial state
$\sm$ such that $\sm ( 1, 0) = \frac{1}{2}$.
Let $d \geq 2$.
Using Lebesgue measure on $[0, 1]^d$, do we have
\[
C ([0, 1]^d) \fr (\C \oplus \C)^{\fr \infty}
  \cong (\C \oplus \C)^{\fr \infty} ?
\]
\end{qst}

The answer to Question~\ref{Q_5Z19_01n}
would be yes if the answer to
the first part of the following question is positive.
The second part would also cover \cref{Qst_6129_half} and many
similar questions.

\begin{qst}\label{Q_6129_aue}
Let $\ta$ be the unique \tst{} on $C ([0, 1])^{\fr \infty}$.
Let $n \in \N$, and let
$\ph, \ps \colon C ([0, 1]^n) \to C ([0, 1])^{\fr \infty}$
be injective unital \hm{s} such that $\ta \circ \ph = \ta \circ \ps$.
Does it follow that $\ph$ and $\ps$
are approximately unitarily equivalent?

If so, more generally, suppose we replace $C ([0, 1])^{\fr \infty}$
with a simple unital separable C*-algebra
with stable rank $1$, strict comparison,
and a unique $2$-quasitracial state~$\ta$ which is a trace.
Does this still hold?
\end{qst}

For $n = 1$, a positive solution is contained
in~\cite[Theorem~1.0.1]{Robert_NCCW} and~\cite[Theorem
4.1]{Ciuperca_Elliott},
a result which plays a key role in this paper.

Going beyond contractible spaces,
we think that the following questions would be natural to consider.

\begin{qst}\label{Q_6129_More}
Consider $S^1$ with Lebesgue measure,
let $X$ be a union of $S^1$ and a line segment connected at a point
with normalized one dimensional Hausdorff measure,
let $Y$ be an annulus in the plane with normalized Lebesgue measure,
and let $S^3$ have normalized surface measure.
\begin{enumerate}
\item\label{5X23_1}
Consider $\C \oplus \C$ with some tracial state~$\sm$.
To ensure that the free products below have real rank zero,
surely the easier case, assume that $\sm (1, 0) \not\in \Q$.
Do we have:
\begin{enumerate}
\item\label{5X23_1_A}
$C (S^1) \fr (\C \oplus \C)^{\fr \infty}
 \cong C (X) \fr (\C \oplus \C)^{\fr \infty}$?
\item\label{5X23_1_BB}
$C (S^1) \fr (\C \oplus \C)^{\fr \infty}
 \cong C (Y) \fr (\C \oplus \C)^{\fr \infty}$?
\item\label{5X23_1_CC}
$C (S^1) \fr (\C \oplus \C)^{\fr \infty}
 \cong C (S^3) \fr (\C \oplus \C)^{\fr \infty}$?
\end{enumerate}
\item\label{5X23_22}
Use Lebesgue measure on $[0, 1]$.
Do we have:
\begin{enumerate}
\item\label{5X23_2_aa}
$C (S^1) \fr C ([0, 1])^{\fr \infty}
 \cong C (X) \fr C ([0, 1])^{\fr \infty}$?
\item\label{5X23_2__b}
$C (S^1) \fr C ([0, 1])^{\fr \infty}
 \cong C (Y) \fr C ([0, 1])^{\fr \infty}$?
\item\label{5X23_2_cc}
$C (S^1) \fr C ([0, 1])^{\fr \infty}
 \cong C (S^3) \fr C ([0, 1])^{\fr \infty}$?
\end{enumerate}
\end{enumerate}
\end{qst}

In each part of the question,
the items appear to be successively harder.
In the first part we have real rank zero.
A union of a circle with a line segment is one dimensional,
so $C (X)$ is semiprojective.
One might hope that there could be an ad hoc argument
to prove such a theorem.
An annulus is homotopy equivalent to a circle, but two dimensional.
The three dimensional sphere has the same $K$-theory as the circle.
However, aside from being of higher dimension,
$C (S^3)$ is not $K_1$-surjective,
and we cannot find maps between $C (S^1)$ and $C (S^3)$
which are nontrivial on $K_1$.
The second part of the question
would add complexity by not assuming real rank zero.

The purely infinite simple case is different,
because there is now never a canonical choice of state,
and because inner automorphisms
need not respect states which are not tracial.
We suggest the questions below.
To keep things simple, we discuss only free products of
infinitely many copies of the same algebra.
The discussion will use the following two results.

\begin{Prop}\label{P_6302_UCT}
Let $A_1$ and $A_2$ be separable \ca{s}, let $B$ be a \ca{}
which is a countable direct sum of \fd{} matrix algebras,
and for $j = 1, 2$ let $\io_j \colon B \to A_j$ be an injective
quasiunital \hm{} and let $E_j \colon A_j \to \io_j (B)$
be a nondegenerate conditional expectation.
Assume that $A_1$ and $A_2$ satisfy the Universal Coefficient Theorem.
Then the reduced free product
$(A_1, E_1) *_{\operatorname{r}, B} (A_2, E_2)$
satisfies the Universal Coefficient Theorem.
\end{Prop}

Here, quasiunital means that $\io_j (B) A_j$ is dense in~$A_j$.
See the discussion before \cite[Lemma~3.6]{Eliasen}.
Nondegenerate means that the associated Gelfand-Naimark-Segal
representation is faithful.
See the discussion before \cite[Theorem~1.1]{Hsgw}.

The case we need is $B = \C$, so that $E_1$ and $E_2$ are just
states with faithful Gelfand-Naimark-Segal representations.
However, it seems useful to give the currently known generality.

The KK-equivalence $G$ in the proof
is used in special cases in \cite{Germain_2} and
\cite{Thn2}, and in this generality in~\cite{Eliasen}.
However, it is not explicit in the statements of theorems in those papers
that this map is a KK-equivalence.

The argument was given in the proof of \cite[Proposition~4.11]{PhlVls2},
but the statement there assumes the states are tracial.

\begin{proof}[Proof of \cref{P_6302_UCT}]
By \cite[Theorem~1.1]{Hsgw}, the canonical map
\[
A_1 *_B A_2 \to (A_1, E_1) *_{\operatorname{r}, B} (A_2, E_2)
\]
is a KK-equivalence.
(This does not require any conditions on $B$ beyond separability.)
Therefore it suffices to prove that $A_1 *_B A_2$
satisfies the Universal Coefficient Theorem.

Define $\io \colon B \to A_1 \oplus A_2$
by $\io (b) = (\io_1 (b), \io_2 (b))$ for $b \in B$.
Let $C \io$ be the mapping cone of~$\io$.
Let $G \colon C \io \to S (A_1 *_B A_2)$ be as defined at the
beginning of \cite[Section~3]{Eliasen}.
The algebra $C \io$ satisfies the Universal Coefficient Theorem,
and, for any~$D$, the algebra $S D$
satisfies the Universal Coefficient Theorem \ifo{} $D$ does.
Therefore it is enough to prove that $G$ is a KK-equivalence.
The proof of \cite[Theorem~4.3]{Eliasen} shows that
conditions (1) and~(2) of \cite[Theorem~3.3]{Eliasen} hold.
The proof of \cite[Theorem~3.3]{Eliasen} proceeds by
showing that $G$ is a KK-equivalence;
see the proof of \cite[Theorem~2.7]{Thn2}.
\end{proof}

We are grateful to Narutaka Ozawa for allowing us to include
the proof of the next proposition.

\begin{Prop}[Ozawa]\label{P_6302_Nonnuc}
Let $A$ be a separable unital \ca, and let $\om$ be a
state on~$A$
 whose associated Gelfand-Naimark-Segal representation is faithful.
As in \cref{Nt_5Z19_FP}, take reduced free products to be
amalgamated over~$\C$.
Then $(A, \om)^{\fr \I}$ is nuclear \ifo{} $A$ is nuclear and $\om$
is pure.
In this case, the infinite reduced free product state is pure.
\end{Prop}

We use the following result.
It should be standard, but we have not found
the full statement in the literature.
The main part is in \cite[Theorem~1.6.5]{VoiDyNi},
and the rest is at least implicit in the proof of Theorem~5 of the
unpublished Master's thesis~\cite{LiQ}.

\begin{Thm}\label{T_6402_CommFP}
Let $( (A_i, \om_i))_{i \in I}$ be a family of unital \ca{s}
equipped with states $\om_i \colon A_i \to \C$.
For $i \in I$, let $\pi_i \colon A_i \to L (H_i)$ be the
Gelfand-Naimark-Segal representation associated to $\om_i$,
let $\xi_i \in H_i$ be the distinguished vector,
set $H_i^{\circ} = (\C \xi_i)^{\perp}$,
set $M_i = \pi_i (A_i)'$, and let $\gm_i \colon M_i \to \C$
be the state
$\gm_i (b) = \langle b \xi_i, \xi_i \rangle$ for $b \in M_i$.
Let $H = \bigast_{i \in I} H_i$ be the free product Hilbert space
(\cite[Definition~1.3.1]{VoiDyNi}), given by
\[
H = \C \xi \oplus \bigoplus_{n = 1}^{\I}
    \left( \bigoplus_{\substack{
    i_1,\ldots,i_n \in I \\
    i_1 \neq i_2,\ i_2 \neq i_3,\ \ldots,\ i_{n-1} \neq i_n
    }}
             H_{i_1}^{\circ} \otimes H_{i_2}^{\circ} \otimes
             \cdots \otimes H_{i_n}^{\circ} \right).
\]
Let $(A, \om) = \bigastr_{i \in I} (A_i, \om_i)$ be the reduced
free product amalgamated over~$\C \cdot 1_{A_i}$ for $i \in I$,
and let $\pi \colon A \to L (H)$ be the free product representation
(\cite[Definition~1.5.1]{VoiDyNi}).
Then $\gm_i$ is a faithful normal state for every $i \in I$.
Moreover, the \pj{} $p \in L (H)$ onto ${\overline{\pi (A)' \xi}}$
is in $\pi (A)''$, and
$p \pi (A)'$ is unitarily equivalent to
the free product representation of the von Neumann algebra
reduced free product $\bigastr_{i \in I} (M_i, \gm_i)$.
\end{Thm}

\begin{proof}
For $i \in I$, define $V_i = {\overline{M_i \xi_i}} \S H_i$,
and define $V_i^{\circ} = V_i \cap (\C \xi_i)^{\perp}$.
Then set $V = \bigast_{i \in I} V_i \S H$.
For $i \in I$, let $\rh_i \colon M_i \to L (H)$ be the
restriction to $M_i$ of the representation $\rh_i$ in
\cite[Definition~1.6.4]{VoiDyNi}.
(There is a misprint in the definition: for $b \in L (H_i)$,
$\rh_i (b)$ should be $W_i (1 \otimes b) W_i^{-1}$.
The formula is correctly given in the earlier paper
\cite[1.2]{Vclc85}:
the operator $W_i$ in \cite{Vclc85} is the inverse of
the operator called $W_i$ in \cite{VoiDyNi}.)
Set $M = \left( \bigcup_{i \in I} \rh_i (M_i) \right)''$.
By \cite[Theorem~1.6.5]{VoiDyNi}, we have $\pi (A)' = M$.
(In~\cite{VoiDyNi}, the algebras $A_i$ are von Neumann algebras,
but this does not affect the conclusion.)

For $i \in I$, the state $\gm_i$ is faithful
because $\xi_i$ is cyclic for $\pi_i (A_i)$.
One checks that ${\overline{M \xi}} = V$.
One further checks that
$\langle \rh_i (b) \xi, \xi \rangle
 = \langle b \xi_i, \xi_i \rangle = \gm_i (b)$
for $i \in I$ and $b \in M_i$.
(Caution: the tensor factors in the domain of $W_i$
before \cite[Definition~1.6.4]{VoiDyNi} are in the opposite order
from those in the domain of $V_i$
in \cite[Definition~1.5.1]{VoiDyNi}.)
Therefore $b \mapsto \rh_i (b) |_{V_i}$ is unitarily equivalent
to the Gelfand-Naimark-Segal representation of $(M_i, \gm_i)$,
preserving the standard cyclic vectors.
It follows that the family $( \rh_i (\cdot) |_V )$ of representations
of $M_i$ on~$V$
is unitarily equivalent to the family of representations of $M_i$
used to construct $\bigastr_{i \in I} M_i$.
Recalling that $p$ is the \pj{} on~$V$,
it follows that $p M \cong \bigastr_{i \in I} M_i$.
This completes the proof.
\end{proof}

\begin{proof}[Proof of \cref{P_6302_Nonnuc}]
For $n \in \N$ let $\om_n$ be the reduced free product state
on $(A, \om)^{\fr n}$.

The reduced free product of two unital nuclear \ca{s}
% (amalgamated over~$\C$, as usual)
is nuclear whenever
the states have faithful Gelfand-Naimark-Segal representations
and one of the states is pure, by \cite[Theorem 4.8.7]{Brown_Ozawa}.
Thus, if $A$ is nuclear and $\om$ is pure, then
$(A, \om)^{\fr n}$ is nuclear for all~$n$ by induction,
and $(A, \om)^{\fr \I}$ is nuclear by taking direct limits.

We claim that if $( (A_i, \om_i))_{i \in I}$ is a family of unital \ca{s}
equipped with pure states $\om_i \colon A_i \to \C$,
then the free product state
$\om$ on $A = \bigastr_{i \in I} (A_i, \om_i)$ is again pure.
The proof is an application of \cref{T_6402_CommFP}.
Using the notation there, we have $M_i = \C \cdot 1$ for all $i \in I$.
It follows from the unitary equivalence statement at the end
of \cref{T_6402_CommFP} that the \pj{} $p$ there has rank~$1$.
Recall that $\xi$ is the standard cyclic vector
for the free product representation $\pi \colon A \to L (H)$.
Since $p \xi = \xi$ and $p \in \pi (A)''$,
we deduce that the rank one \pj{} on $\C \xi$ is in $\pi (A)''$.
Since $\xi$ is cyclic, all rank one \pj{s} are in $\pi (A)''$.
Therefore $\pi$ is irreducible.
The claim follows.

We now assume that $(A, \om)^{\fr \I}$ is nuclear.

We first claim that for any $(A_1, \rh_1)$ and $(A_2, \rh_2)$
for which the Gelfand-Naimark-Segal representations are faithful,
and regarding $A_1$ and $A_2$ as subalgebras of $A_1 \fr A_2$,
there is a conditional expectation $E$ from $A_1 \fr A_2$ onto~$A_1$
such that if $a \in A_2$ then $E (a) = \rh_2 (a) \cdot 1$.
This is surely known, and is immediate from a known more general
result, but we don't know a direct reference.
To prove the claim, in \cite[Theorem 4.8.5]{Brown_Ozawa} take
\[
D = \C, \qquad
E_1^{A_1} = \rh_1, \qquad
E_2^{A_1} = \rh_2, \qquad
B_1 = A_1, \qquad
B_2 = \C,
\]
\[
E_1^{B_1} = \rh_1, \qquad
E_2^{B_2} = \id_{\C}, \qquad
\te_1 = \id_{A_1}, \andeqn
\te_2 = \rh_2.
\]
The result is a unital completely positive map
$E \colon A_1 \fr A_2 \to A_1$
such that $E (a) = a$ for $a \in A_1$ and $E (a) = \rh_2 (a) \cdot 1$
for $a \in A_2$.
It is a conditional expectation by \cite[Theorem 1.5.10]{Brown_Ozawa}.
This proves the claim.

In particular, for every $n \in \N$ we get a conditional expectation
$E_{n, n + 1} \colon (A, \om)^{\fr (n + 1)} \to (A, \om)^{\fr n}$.
Composition gives conditional expectations
$E_{m, n} \colon (A, \om)^{\fr n} \to (A, \om)^{\fr m}$
for $m, n \in \N$ with $m \leq n$ such that,
if also $k \leq m$, then $E_{k, m} \circ E_{m, n} = E_{k, n}$.
Taking direct limits, for $m \in \N$
there is a conditional expectation
$E_{m, \I} \colon (A, \om)^{\fr \I} \to (A, \om)^{\fr m}$.

Using the completely positive approximation property,
it is immediate that if $A$ is nuclear and $E \colon A \to B$ is a
conditional expectation onto a subalgebra, then $B$ is nuclear.
(This also follows from
\cite[Exercise 2.1.3 and Exercise 2.1.5]{Brown_Ozawa}.)
Nuclearity of $(A, \om)^{\fr \I}$ therefore implies nuclearity
of $(A, \om)^{\fr n}$ for all $n \in \N$.
In particular, $A$ is nuclear.

We finish the proof by showing that if $\om$ is not pure then
$(A, \om)^{\fr \I}$ is not nuclear.
For any state $\rh$ on a \ca~$D$, let $\pi_{\rh}$
% $\pi_{\rh} \colon D \to L (H_{\rh})$
denote the Gelfand-Naimark-Segal representation.
Since $\om$ is not pure, $\dim ( \pi_{\om} (A)' ) \geq 2$.
\cref{T_6402_CommFP} implies that $\pi_{\om_2} (A \fr A)'$
has a cutdown which is the von Neumann algebra reduced free
product of $\pi_{\om} (A)'$ with itself.
So $\dim ( \pi_{\om_2} (A \fr A)' ) = \I$.
Again using $\dim ( \pi_{\om} (A)' ) \geq 2$,
apply \cite[Theorem~4.1]{Ueda} (see the notation and
standing hypotheses at the beginning of \cite[Section~4]{Ueda})
to deduce that
the diffuse part of the von Neumann algebra reduced free product
of $\pi_{\om} (A)'$ and $\pi_{\om_2} (A \fr A)'$ is nontrivial,
and \cite[Corollary~4.3]{Ueda} to see that
the diffuse part is prime.
Since hyperfinite factors with separable predual are not prime,
this von Neumann algebra reduced free product is not hyperfinite.
Using \cref{T_6402_CommFP} again, in the same way as above,
we see that $\pi_{\om_3} (A \fr A \fr A)'$ is not hyperfinite.
So also $\pi_{\om_3} (A \fr A \fr A)$ is not hyperfinite,
whence $A \fr A \fr A$ is not nuclear.
Since there is a conditional expectation from $(A, \om)^{\fr \I}$
to $A \fr A \fr A$,
it follows that $(A, \om)^{\fr \I}$ is not nuclear.
\end{proof}

\begin{qst}\label{Q_6130_DiffStates}
Let $n \in \{ 2, 3, \ldots, \I \}$.
Let $\om_1$ and $\om_2$ be states on~${\mathcal{O}}_{n}$.
Write $( {\mathcal{O}}_{n}, \om_j)^{\fr \I} = (D_j, \rh_j)$.
When is $D_1 \cong D_2$?
When is $(D_1, \rh_1) \cong (D_2, \rh_2)$
(state preserving isomorphism)?
\end{qst}

The most natural cases to consider are $n = 2$ and $n = \I$.

\cref{P_6302_Nonnuc} tells us that there are two cases to consider:
$\om_1$ and $\om_2$ are both pure,
and $\om_1$ and $\om_2$ are both not pure.
If $\om_1$ and $\om_2$ are both pure, then,
by \cite[Theorem 1.1]{KshOzwSak},
there is $\af \in \Aut ({\mathcal{O}}_{n})$
(in fact, $\af$ can be chosen to be approximately inner) such that
$\om_2 = \om_1 \circ \af$.
Therefore in fact $(D_1, \rh_1) \cong (D_2, \rh_2)$.
But what happens if both $\om_1$ and $\om_2$ are not pure?

In \cref{Q_6130_DiffStates}
we wrote $\mathcal{O}_{n}$ for concreteness, but
the same discussion applies to any Kirchberg algebra
which satisfies the Universal Coefficient Theorem.

\begin{qst}\label{Q_6130_AbsOI}
Let $C$ be a purely infinite simple separable \uca{}
and let $\om$ be a faithful state on~$C$.
Let $\sm$ be a state on~${\mathcal{O}}_{\I}$.
Is it true that ${\mathcal{O}}_{\I} \fr C^{\fr \I} \cong C^{\fr \I}$?
Is there an isomorphism which preserves the free product states?
\end{qst}

\begin{qst}\label{Q_6130_AbsC01}
Let $C$ be a purely infinite simple separable \uca{}
and let $\om$ be a faithful state on~$C$.
Let $\rh$ be the tracial state on $C ([0, 1])$
induced by Lebesgue measure.
Is it true that $C ([0, 1]) \fr C^{\fr \I} \cong C^{\fr \I}$?
Is there an isomorphism which preserves the free product states?
\end{qst}

In both \cref{Q_6130_AbsOI} and \cref{Q_6130_AbsC01},
both algebras have the same K-theory.
To see this in \cref{Q_6130_AbsOI},
use the case $B = \C$ of \cite[Theorem 6.4]{Thn2} to see that
$K_* ({\mathcal{O}}_{\I} * C^{\fr \I}) \cong K_* (C^{\fr \I})$
(full free product amalgamated over~$\C$ on the left), and then use
the case $B = \C$ of \cite[Theorem 1.1]{Hsgw} to get
$K_* ({\mathcal{O}}_{\I} \fr C^{\fr \I}) \cong K_* (C^{\fr \I})$.
The proof for \cref{Q_6130_AbsC01} is the same.
Also, both algebras are purely infinite and simple,
by the last statement in \cite[Theorem 2.1]{Dykema_Rordam_II}.

Suppose further that $C$
is nuclear and satisfies the Universal Coefficient Theorem.
In both questions,
both algebras then satisfy the Universal Coefficient Theorem
by \cref{P_6302_UCT}.
If $\om$ is pure, then in both questions
we claim that the algebras are isomorphic.
Indeed, the infinite reduced free product $(C, \om)^{\fr \I}$ is nuclear
by \cref{P_6302_Nonnuc}, and then
${\mathcal{O}}_{\I} \fr C^{\fr \I}$ and $C ([0, 1]) \fr C^{\fr \I}$
are nuclear by \cite[Theorem 4.8.7]{Brown_Ozawa}
and the pureness statement in \cref{P_6302_Nonnuc}.
Isomorphism now follows from the Kirchberg-Phillips classification.
But when $\om$ is not pure,
both questions are already interesting in this case.
Indeed, they are interesting when $C = {\mathcal{O}}_{\I}$ and, in
\cref{Q_6130_AbsOI}, the states $\om$ and $\sm$ are not in the
same orbit under the action of $\Aut ({\mathcal{O}}_{\I})$
on the state space.

One key difficulty is as follows.
In \cref{Q_6130_AbsOI}, for example,
we know by \cite[Theorem 3.3]{LinPhl2}
that if $D$ is a purely infinite simple unital \ca,
then any two unital \hm{s}
$\ph_1, \ph_2 \colon {\mathcal{O}}_{\I} \to D$
are approximately unitarily equivalent.
(See \cite[Theorem~B]{BwnGab} for a more general result.)
This isn't good enough.
In the proof of \cref{thm_main},
which has a unital direct limit $A$ of one dimensional NCCW complexes
in place of ${\mathcal{O}}_{\I}$,
we needed $\af \in \Aut (D)$ (with $D = A \fr C$)
such that $\af \circ \ph_1$ is close to $\ph_2$ on a large finite set,
and such that, for example, $\af \fr \id_C$ is well defined.
Thus, $\af$ must preserve the state on~$D$.
Preservation of the state is true there
because $\af$ is inner and the state is tracial,
but in general we know no way to be sure that this happens.
This difficulty suggests the following problem.

\begin{qst}\label{Q_6130_aueSt}
Let $D$ be a unital purely infinite simple separable \ca,
let $\mu$ be a state on~$D$,
and let $\ph_1, \ph_2 \colon {\mathcal{O}}_{\I} \to D$ be unital \hm{s}
such that $\mu \circ \ph_1 = \mu \circ \ph_2$.
Is it true that for every $\ep > 0$ and every finite subset
$F \SQ {\mathcal{O}}_{\I}$, there is $\af \in \Aut (D)$
such that $\| \ph_2 (a) - \af (\ph_1 (a)) \| < \ep$ for all $a \in F$,
and also $\mu \circ \af = \mu$?
If so, can $\af$ be taken to be inner?
\end{qst}

One can ask similar questions about reduced free products $C^{\fr \I}$
when $C$ is stably finite but $C^{\fr \I}$ is purely infinite
because the state used on~$C$ is not tracial.

\bibliographystyle{plain}
\bibliography{free_bib}

\end{document}